\documentclass[11pt]{amsart}

\usepackage[utf8]{inputenc}
\usepackage{fullpage, mathrsfs, mathtools, amssymb, tikz,bm, amscd, scrextend, tikz-cd, verbatim, enumerate}
\usepackage[colorlinks]{hyperref}
\usepackage[inline]{enumitem}
\usepackage{microtype}
\usepackage{verbatim}

\usetikzlibrary{matrix, arrows}

\input xy
\xyoption{all} 

	\newtheorem{theorem}{Theorem}[section]
	\newtheorem{lemma}[theorem]{Lemma}
	\newtheorem{prop}[theorem]{Proposition}
	\newtheorem{cor}[theorem]{Corollary}

	\theoremstyle{definition}
	\newtheorem{definition}[theorem]{Definition}
\newtheorem{example}[theorem]{Example}
\newtheorem{remark}[theorem]{Remark}

\theoremstyle{remark}

\newcounter{step}

\newcommand{\G}{{\mathbb G}}
\newcommand{\Q}{{\mathbb Q}}

\newcommand{\bP}{{\mathbb P}}
\newcommand{\bA}{{\mathbb A}}
\newcommand{\bQ}{{\mathbb Q}}

\newcommand{\calK}{\mathcal K}
\newcommand{\calO}{\mathcal O}
\newcommand{\calM}{\mathcal M}

\newcommand{\calF}{\mathcal F}

\newcommand{\calA}{\mathcal A}
\newcommand{\calB}{\mathcal B}

\newcommand{\calW}{\mathcal W}
\newcommand{\calE}{\mathcal E}
\newcommand{\calU}{\mathcal U}
\newcommand{\calX}{\mathcal X}
\newcommand{\calC}{\mathcal C}

\newcommand{\calR}{\mathcal R}

\newcommand{\calD}{\mathcal D}

\newcommand{\cEe}{\mathcal R(1/12 + \epsilon)}

\newcommand{\wi}{\mathrm{W}_{\mathrm{I}}}
\newcommand{\wii}{\mathrm{W}_{\mathrm{II}}}
\newcommand{\wiii}{\mathrm{W}_{\mathrm{III}}}

\newcommand{\Spec}{\mathrm{Spec}}
\newcommand{\Proj}{\mathrm{Proj}}
\newcommand{\SL}{\mathrm{SL}}

\newcommand{\mb}[1]{\mathbb{#1}}

\newcommand{\BB}{ (_{\Gamma}\!\backslash\!^{\mathbb{B}})^*}
\newcommand{\BBO}{ (_{\Gamma_0}\!\backslash\!^{\mathbb{B}_0})^*}

\title[Moduli of degree one del pezzo surfaces]{Stable pair compactifications of the moduli space of degree one del Pezzo surfaces via elliptic fibrations}
\author{Kenneth Ascher \& Dori Bejleri}

\begin{document}

\maketitle

\section{Introduction}

A smooth del Pezzo surface of degree $d$ is a smooth projective surface with $-K_X$ ample and $K_X^2 = d$. The goal of this paper to construct geometric compactifictions of the moduli space of degree one del Pezzo surfaces. Here geometric means that the moduli space carries a universal family of possibly degenerate del Pezzo surfaces. 

It is natural to study a del Pezzo surface via its anticanonical linear series $|-K_X| : X \dashrightarrow \mb{P}^d$. For $d \geq 3$ this is a closed embedding and for $d = 2$ it is a double cover. However, for $d = 1$ the anticanonical pencil is not a morphism; it has a unique basepoint. The blowup of $X$ at this basepoint is a rational elliptic surface $Y \to \mb{P}^1$ with section given by the exceptional divisor. Equivalently, $X$ may be obtained as the blowup of $\mb{P}^2$ at $8$ points in general position, and the anticanonical pencil is the unique pencil of cubics passing through these points. By the Cayley-Bacharach theorem, there is a unique $9^{th}$ point in the base locus of this pencil which becomes the basepoint of $|-K_X|$. Our strategy is to use the structure of the elliptic fibation $Y \to \mb{P}^1$ to construct a geometrically meaningful compactification of the space of degree one del Pezzo surfaces.

\begin{theorem}[see Sec. \ref{sec:dp} and Fig. \ref{fig:moduli}] There exists a proper Deligne-Mumford stack $\calR = \calR(1/12 + \epsilon)$ parametrizing anti-canonically polarized broken del Pezzo surfaces of degree one with the following properties:
\begin{itemize}
\item The interior $\calU \subset \calR$ parametrizes degree one del Pezzo surfaces with at worst rational double point singularities. 
\item The complement $\calR \setminus \calU$ consists of a unique boundary divisor parametrizing $2$-Gorenstein semi-log canonical surfaces with ample anticanonical divisor and exactly two irreducible components. 
\item The locus $\calR^\circ \subset \calR$ parametrizing surfaces such that every irreducible component is normal is a smooth Deligne-Mumford stack.  \end{itemize}
\end{theorem}

\begin{figure}[!h]
\includegraphics[scale=.5]{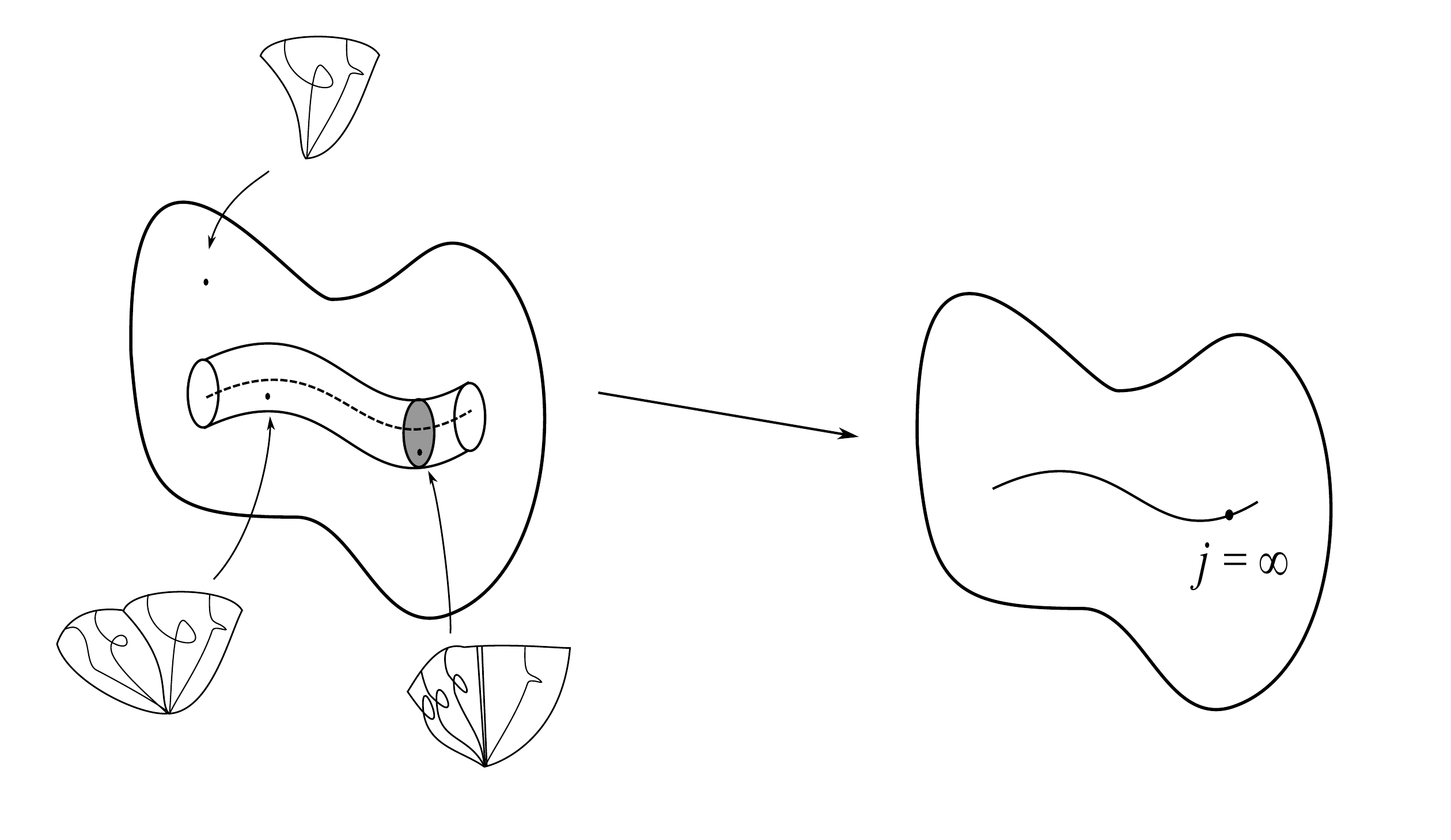}
\caption{The space on the right is Miranda's GIT compactification $W$, and the space on the left is $\calR(1/12 + \epsilon)$. Note that these spaces are birational.  The strictly semistable locus of $W$ is a rational curve with a special point $(j = \infty)$. The ``tube" depicts the locus of $\calR(1/12 + \epsilon)$ parametrizing non-normal surfaces, which is fibered over the strictly semistable locus of $W$. The shaded in piece depicts the fiber over $j=\infty$. The surfaces parametrized by a generic point of each stratum of $\calR(1/12+ \epsilon)$ are depicted. }\label{fig:moduli}
\end{figure}

We provide an explicit description of the surfaces parametrized by the boundary of $\calR(a)$ for $a \le 1/6$ in Theorems \ref{thm:dptypes} and \ref{thm:dpinfinity}. 

The space $\calR(1/12 + \epsilon)$ is the last in a sequence of spaces $\calR(a)$ for $1/12 < a \le 1$ arising as compactifications of the space of rational elliptic surfaces. The existence of these spaces follows from the machinery developed in a series of papers (\cite{calculations, tsm, master}) where the authors study the birational geometry of the moduli space of \emph{weighted stable elliptic surfaces} -- moduli spaces parametrizing pairs of elliptic fibrations with section and marked fibers that are stable pairs in the sense of KSBA. Taking inspiration from Hassett's moduli spaces of weighted stable curves, one can construct proper Deligne-Mumford moduli stacks $\calE_{\calA}$ where $\calA = (a_1,\ldots, a_n)$ is a weight vector keeping track of the weights of the marked fibers.  As one varies $\calA$, the moduli spaces and their universal families are related by divisorial contractions and flips (see \cite{master}). In constructing the spaces $\calR(a)$ above, we give an in depth analysis of the moduli spaces $\calE_{\calA}$ and wall crossings morphisms between them in the case of rational elliptic surfaces.

A generic rational elliptic surface has 12 nodal fibers $F_i$, and so we consider a slice of the moduli space, denoted $\calE^s_{\calA}$, which compactifies the space of pairs $(f : X \to C, S + F)$, where $X$ is a rational elliptic surface with section $S$ and $F = \sum a_iF_i$. For most of this paper we will focus on the case $\calA = (a, \dots, a)$ is a constant  vector and in this case we can define the space $\calR(a) = \calE^s_\calA / S_{12}$. When $a \le 1/6$, the section of every surface parametrized by $\calR(a)$ must be contracted to form the \emph{pseudoelliptic surfaces} of La Nave \cite{ln} (see also Definition \ref{def:pseudo1}) which in the case of rational elliptic surfaces are exactly the degree one del Pezzo surfaces.  

\subsection{Connnection to GIT compactifications and period mappings} One may associate to any elliptic surface with section a \emph{Weierstrass equation} $y^2 = x^3 + Ax + B$ where $A$ and $B$ are sections of line bundles on the base curve. This is the equation cutting out the \emph{Weierstrass model} (see Definition \ref{def:weierstrassmodel}) obtained by contracting all fibral components that do not meet the section. In the case of rational elliptic surfaces, $A$ and $B$ are degree $4$ and $6$ homogeneous polynomials on $\mb{P}^1$. Using this Weierstrass data, Miranda \cite{mir} constructed a GIT compactification of the moduli space of rational elliptic surfaces (see Section \ref{sec:GIT}). We show that our compactification is a certain blowup of the GIT compactification along the strictly semi-stable locus: 

\begin{theorem}[see Theorem \ref{thm:egit}] Let $R = R(1/12 + \epsilon)$ be the coarse moduli space of $\calR(1/12 + \epsilon)$ and $\Delta \subset R$ the boundary divisor parametrizing non-normal surfaces with $U = R \setminus \Delta$. There is a morphism $R \to W$ to Miranda's GIT compactification such that the following diagram commutes.
$$
\begin{tikzcd}
\Delta \arrow[r,hook] \arrow{d}[swap]{j} & R \arrow[d] & U \arrow[d,"\cong"] \arrow[l,hook'] \\ 
\mb{P}^1 \arrow[r] & W & W^s \arrow[l,hook']
\end{tikzcd}
$$
Here $\Delta \to \mb{P}^1$ sends the surface $X \cup Y$ to the $j$-invariant of the double locus. Then $\mb{P}^1 \to W^{sss} \subset W$ maps bijectively onto the strictly semistable locus, and $U \to W^s$ is an isomorphism of the interior of $R$ with the GIT stable locus.
\end{theorem}

Now let $\calD^*$ denote the GIT compactification of the space of 12 points in $\bP^1$ up to automorphism. To a Weierstrass equation $y^2 = x^3 + Ax + B$, one may associate a discriminant $\mathscr{D} = 4A^3 + 27B^2$. This gives a rational map $W \dashrightarrow \calD^*$ and it is natural to ask how close this map is to a morphism. Indeed it is easy to see that it cannot extend to all of $W$ (for example the Weierstrass equation of a surface with an $I_7$ fiber is GIT stable but its discriminant is not GIT semistable). We prove that this rational map can be understood by the KSBA compactifications $R(a)$:

\begin{theorem}[see Corollary \ref{cor:gitbasecurve}] There is a morphism $R(1/6) \to \calD^*$ resolving $W \dashrightarrow \calD^*$.  \end{theorem}

In \cite{hl}, Heckman and Looijenga study a certain period mapping for rational elliptic surfaces. In particular, they show that the moduli space of $12\mathrm{I}_1$ rational elliptic surfaces is locally a complex hyperbolic variety and identify the rational map $W \dashrightarrow \calD^*$ as induced by the period mapping. Furthermore they describe the normalization of the image in $\calD^*$ as the Satake-Baily-Borel (BB) compactification $\mathscr{M}^*$ of a ball quotient, and compare the boundary strata of this compactification with the GIT construction of Miranda discussed above by introducing a space $W^*$ which dominates both. Note that neither $W, \mathscr{M}^*$, nor $W^*$ carry universal families of surfaces, in the sense that they are not coarse moduli spaces of a proper Deligne-Mumford stack with a universal elliptic surface, which is a source of difficulty in studying their boundary strata. 

\begin{theorem}[See Theorem \ref{thm:hlcompare}] There is a birational morphism $R(1/6) \to W^*$ of coarse moduli spaces such that $\calR(1/6)$ is the minimal proper Deligne-Mumford stack above both $\mathscr{M}^*$ and $W$ extending the universal family on $\mathscr{M}$. \end{theorem}

Finally, we note that recent exciting work of Laza-O'Grady (see \cite{ol, lo2, lo3}) introduces the ``Hassett-Keel-Looijenga" program which aims to use a $\Proj$ of canonical ring construction to interpolate between GIT and BB compactifications of moduli spaces of surfaces, generalizing earlier ideas of Looijenga. In the final section of this paper, we suggest a similar approach:  use KSBA compactifications to connect GIT and BB compactifications. Not only do we expect that KSBA compactifications will map to the spaces appearing in their construction, but KSBA will also shed light on understanding the boundary more explicitly, as the boundaries produced by GIT and BB compactifications can be quite mysterious (see Section \ref{sec:observations} for more details). \\

\subsection{Further results and discussion}
Let us now make a few remarks to put these results in context. As mentioned above, we provide an explicit description of the surfaces parametrized by the boundary of $\calR(a)$ for $a \le 1/6$ (Theorems \ref{thm:dptypes} and \ref{thm:dpinfinity}). To do this we give a complete description of the wall and chamber structure of the domain of admissible weights $1/12 < a \le 1$, and in fact the $a = 1/12 + \epsilon$ used above can be taken to be any any $a$ in the lowest chamber $1/12 < a < 1/10$. 

The determination of the wall and chamber structure and birational contractions of $\calR(a)$ across walls follows from an in depth understanding of stable reduction process of running twisted stable maps, and lowering weights to the desired $\calA$. This idea goes back to Abramovich-Vistoli \cite{av} and is used by La Nave in \cite{ln}. The general theory is provided in \cite{calculations, tsm, master}, and part of the purpose of this paper is to serve as proof of concept for \emph{loc. cit.}

In Section \ref{sec:walls}, we explicitly calculate the walls appearing in the domain of admissible weights for $\calR(a)$. In Section \ref{sec:ksba} we recall the necessary background and results from the original papers as we aim to keep this paper largely self-contained. Our main result regarding walls is as follows (see Definition \ref{def:walltypes} for the definition of $\mathrm{W}_{\mathrm{III}}$).

\begin{theorem}[see Theorems \ref{thm:type1walls} and \ref{thm:jinfwalls} and Proposition \ref{prop:typeii6}] The walls of type $\wiii$ where a (pseudo)elliptic component of the surfaces parametrized by $\calR(a)$ contracts occur at 
\begin{enumerate}
\item $a = 1/k$ for $2 \le k \le 9$, and
\item $a = a_0/k$ where $ 2 \le k \le 5$ and $a_0 \in \{\frac{5}{6}, \frac{3}{4}, \frac{2}{3}, \frac{1}{2}\}$
\end{enumerate}\end{theorem}

Finally, to prove that $\calR^\circ$ is smooth, we utilize Hacking's study of $\bQ$-Gorenstein deformations in \cite{hacking} (see Section \ref{sec:hacking} for a background on Hacking's work, and Section \ref{sec:smooth} for the proof that $\calR^\circ$ is smooth).

\subsection{Related work} Several moduli spaces of del Pezzos have been studied previously using GIT \cite{ishii}, root lattices \cite{seki1, seki2}, pencils of quadrics \cite{hkt1} (degree $4$ case), stable pairs and tropical geometry \cite{hkt} (degree $d \geq 2$ case, and the inspirational for our title), and Gromov-Hausdorff limits \cite{odaka}. We note that in \cite{adesurfaces}, Alexeev and Thompson construct a stable pair compactification for the moduli space of rational elliptic surfaces with a chosen nodal fiber. 

\subsection{Future work} In future work, we will continue the story by studying the boundary geometry of $\calR(1/12 + \epsilon)$. We will explore the Kirwan desingularization \cite{kirwan} of Miranda's GIT quotient, and relate this with $\calR(a)$. In addition, we plan to discuss connections of our work with the Gromov-Hausdorff limit approach of Odaka-Spotti-Sun \cite{odaka}. These are examples of the general philosophy that birational compactifications of moduli spaces of surfaces should be related by a sequence of wall crossings of KSBA moduli spaces of stable pairs as one varies the coefficients of the boundary divisor.

\subsection*{Acknowledgments} We benefited from conversations with Dan Abramovich, Valery Alexeev, Kristin DeVleming, Patricio Gallardo, Brendan Hassett, Giovanni Inchiostro, S\'andor Kov\'acs, Gabriele La Nave, Radu Laza, Eduard Looijenga, and Amos Turchet. Research of K.A. supported in part by an NSF Postdoctoral Fellowship. Research of D.B. is supported in part by funds from NSF grant DMS-1500525 (P.I. Dan Abramovich). K.A. thanks the Mathematics Department at the University of Washington for a pleasant visit during which this work was carried out.
 
\section{Elliptic surfaces} We begin with a review of the geometry of rational elliptic surfaces Our discussion of elliptic surfaces follows \cite{master}, and is largely influenced by \cite{mir3}.

\begin{definition}
An irreducible \textbf{elliptic surface with section} ($f: X \to C, S)$ is an irreducible surface $X$ together with a surjective proper flat morphism $f: X \to C$ to a smooth curve $C$ and a section $S$ such that:
\begin{enumerate} 
\item the generic fiber of $f$ is a stable elliptic curve, and
\item the generic point of the section is contained in the smooth locus of $f$. 
\end{enumerate}
We call the pair $(f: X \to C, S)$ \textbf{standard} if all of $S$ is contained in the smooth locus of $f$.
\end{definition}

Note that we only require that the generic fiber is a \emph{stable} elliptic curve. 

\begin{definition} A surface is \textbf{semi-smooth} if it only has 2-fold normal crossings (locally $x^2 = y^2$) or pinch points (locally $x^2 = y^2z$). A \textbf{semi-resolution} of a surface $X$ is a proper map $g: Y \to X$ such that $Y$ is semi-smooth and $g$ is an isomorphism over the semi-smooth locus of $X$. \end{definition}

Recall that a surface is \emph{relatively minimal} if it is semi-smooth and there are no $(-1)$-curves in any fiber. In particular any relateively minimal elliptic surface with section is standard, and so there are finitely many fiber components not intersecting the section. Contracting these yields an elliptic surface with all fibers reduced and irreducible.

\begin{definition}\label{def:weierstrassmodel} A \textbf{minimal Weierstrass fibration} is an elliptic surface obtained from a relatively minimal elliptic surface by contracting all fiber components not meeting the section. We call the output of this process a \textbf{Weierstrass model}. \end{definition}

The geometry of an elliptic surface is largely influenced by the \emph{fundamental line bundle} $\mathscr L$.

\begin{definition} The \textbf{fundamental line bundle} of a standard elliptic surface $(f: X \to C, S)$ is $\mathscr L := (f_* \mathscr{N}_{S/X})^{-1}$, where $\mathscr{N}_{S/X}$ denotes the normal bundle of $S$ in $X$. For an arbitrary elliptic surface we define $\mathscr{L}$ as the line bundle associated to its semi-resolution. \end{definition}

We remark that since $\mathscr{N}_{S/X}$ only depends on a neighborhood of $S$ in $X$, the line bundle $\mathscr{L}$ is invariant under taking a semi-resolution or Weierstrass model of a standard elliptic surface.  Furthermore, we note that $\mathscr{L}$ enjoys many nice properties. In particular, $\deg(\mathscr{L})|_C \geq 0$, is independent of choice of section $S$, and determines the canonical bundle of $X$. 

If $(f: X \to C, S)$ is a smooth relatively minimal elliptic surface, then $f$ has finitely many singular fibers which are each unions of rational curves with possibly non-reduced components. Recall that the dual graphs are ADE Dynkin diagrams. Furthermore, the possible singular fibers have been classified by Kodaira-Ner\'on. We refer the reader to \cite[Table 1]{calculations} for the complete classification. However, we point out the definition of the fiber types $\mathrm{N}_k$ for $k = 0,1, 2$, which appear on elliptic surfaces with nodal generic fiber and arise when studying slc surfaces (see \cite[Section 5]{calculations}). 

\begin{definition}\label{def:nk} The fibers $\mathrm{N}_k$ are the slc fiber types with Weierstrass equation $y^2 = x^2(x-t^k)$ for $k = 0, 1, 2$. \end{definition}

In the sequel the following surfaces appear: 

\begin{definition}\label{def:pseudo1} A \textbf{pseudoelliptic surface} is a surface $Z$ obtained by contracting the section of an irreducible elliptic surface pair $(f: X \to C, S)$. For any fiber of $f : X \to C$, we call its pushforward to $Z$ a \textbf{pseudofibers} of $Z$. We call $(f: X \to C, S)$ the associated elliptic surface to $Z$. \end{definition}

\subsection{Rational elliptic surfaces}

We are now ready to define when an elliptic surface is \emph{rational}. We note that there are many equivalent definitions. We will define them as follows, and then give some references and discuss equivalent notions.

\begin{definition} We say that an irreducible elliptic surface with section $(f:X \to C, S)$ is \textbf{rational} if $C \cong \bP^1$ and $\deg(\mathscr{L}) = 1$. \end{definition} 

The fact that this definition characterizes rational elliptic surfaces is the content of Lemma III.4.6 of \cite{mir3}. It turns out that all rational elliptic surfaces arise as the blow up of the base locus of a pencil of cubic curves inside $\bP^2$ (see Lemma IV.1.2 of \cite{mir3}). 

\begin{remark}Let $C_1$ and $C_2$ be two (distinct) smooth cubic curves in $\bP^2$. Then the pencil generated by these curves has 9 base points, and blowing up these 9 points in $\bP^2$ gives a morphism $\pi: X \to \bP^1$, where $X$ is a (fibered) rational surface with fibers elliptic curves. In particular it is relatively minimal. Moreoever, the canonical class of $X$ is $-C_1$ and $K_X^2 = 0$. The section $S \subset X$ is given by the last exceptional divisor. In this case, it is clear that $S^2 = -1$ and so $\deg(\mathscr{L}) = 1$ (so that $\calO(1) \cong \mathscr{L}$). The fact that all rational elliptic surfaces are the blowup of $\bP^2$ at the base points of a pencil of generically smooth cubic curves is the content of Lemma IV.1.2 of \cite{mir3}.\end{remark}

Recall from Definition \ref{def:weierstrassmodel} the notion of a \emph{Weierstrass fibration}. It turns out (see Section II.5 of \cite{mir3}) that a rational elliptic surface is defined (locally) by a Weierstrass form: $y^2 = x^3 + Ax + B$, where $A$ and $B$ are sections of $\calO(4)$ and $\calO(6)$ respectively, and the \emph{discriminant} $\mathscr{D} = 4A^3 + 27B^2$ is a section of $\mathscr{L}^{\otimes 12} \cong \calO(12)$ which is not identically zero.

\begin{remark}\label{rmk:singfibers} In fact, since the number of singular fibers of a Weierstrass fibration over a projective curve $C$ is given by $12\deg(\mathscr{L}) = 12\deg(\calO(1)) = 12$ counted properly (see Lemma II.5.7 of \cite{mir3}), a rational elliptic surface has generically 12 (nodal) singular fibers. In this context, \emph{counting properly} means that the singular fiber is weighted by the order of vanishing of the discriminant. Equivalently, the discriminant is a degree $12$ divisor of the base rational curve. \end{remark}

\subsection{Rational elliptic surfaces and degree one del Pezzo surfaces}

\begin{definition}\label{def:dp} Recall a \textbf{degree $n$ del Pezzo surface} is a surface $X$ with at worst canonical singularities such that $-K_X$ is ample and $K_X^2 = n$. \end{definition} 

\begin{remark} It follows by Castelnuovo's Theorem that a del Pezzo surface is necessarily rational. \end{remark} 

Given a degree one del Pezzo surface, the anticanonical linear series $|-K_X| : X \dashrightarrow \mb{P}^1$ has a unique base point $p$. Blowing up along $p$ resolves the basepoint producing a morphism $f: Y = Bl_p(X) \to \mb{P}^1$ with section $S$ the exceptional divisor. The fibers of $f$ are necessarily $K_Y$-trivial curves. It follows by the adjunction formula that $f$ is a genus one fibration with section $S$, i.e. $(f: Y \to \mb{P}^1, S)$ is a rational elliptic surface. 

Conversely, given a rational elliptic surface $(f : Y \to \mb{P}^1,S)$ with at worst rational double point singularities and all fibers being irreducible, e.g. having only twisted or Weierstrass fibers (see Definition \ref{def:fibertypes}), we may blow down the section to obtain a pseudoelliptic surface $X$. By Kodaira's canonical bundle formula, one can check that the pseudofiber class $f$ is ample and linear equivalent to $-K_X$ so that $X$ is a degree one del Pezzo surface. 

This relation between rational elliptic surfaces and degree one del Pezzo surfaces is the main idea behind our construction of the space $\calR(1/12 + \epsilon)$ (see Definition \ref{def:Ra} and Section \ref{sec:dp}) compactifying the moduli space of degree one del Pezzo surfaces. 

\begin{remark} One can obtain a degree one del Pezzo surface $X$ by blowing up $\mb{P}^2$ in $8$ (possibly infinitely near) points and then contracting $(-2)$-curves. By the Cayley-Bacharach theorem, there exists a unique pencil of cubics in $\mb{P}^2$ through these $8$ points that passes through a unique $9^{th}$ point $p$. This becomes the anticanonical pencil of $X$ with basepoint $p$. \end{remark}

\section{Preliminaries on twisted stable maps} 

In this section we review some facts about twisted stable maps that we will use to determine the limits of families of elliptic surfaces. Twisted stable maps are used to compute degenerations of elliptic surface pairs with all coefficients $1$. For more detail see \cite{av, av2, tsm}. 

Briefly, a twisted stable map to a proper Deligne-Mumford stack $\calM$ with projective coarse space $M$ is a representable  morphism $(\calC, \Sigma^\calC) \to \calM$ where
\begin{enumerate}
    \item $\calC$ is an orbifold curve with trivial generic stabilizer whose coarse space $C$ is a nodal curve,
    \item $\Sigma^\calC \subset \calC$ is a collection of marked points in the smooth locus including the smooth points with nontrivial stabilizers,
    \item The map $(C, \Sigma^C) \to M$ of coarse spaces is a stable map in the sense of Konsetvich. 
\end{enumerate}
There is a proper moduli stack of twisted stable maps of fixed degree \cite{av2} and it can be used to construct stacks of fibered surfaces in the case where the target $\calM = \overline{\calM}_{g,n}$ \cite{av, tsm}. 

Relevant for us is the space of twisted stable maps to $\overline{\calM}_{1,1}$ inducing a degree $12$ map on coarse spaces.  Indeed given a rational elliptic surface $(f : X \to \mb{P}^1,S + F)$ with only $\mathrm{I}_1$ singular fibers all of which are marked with coefficient one, there is a morphism $\mb{P}^1 \to \overline{\calM}_{1,1}$ and we can understand degenerations of the surface by degenerating in the space of twisted stable maps.  We list the main features of such degenerations that will be useful for us. 
\subsubsection{TSM Conditions} \label{tsmconditions}
\begin{enumerate}
    \item When the source orbicurve degenerates to a nodal curve, the stabalizer group at the node must act by dual weights on the two branches (see \cite[Definition 3.2.4]{av}). 
    \item When marked points with trivial stabilizer collide, a rational curve carrying the markings sprouts off and maps by a constant map to the target. 
    \item When a marked point with trivial stabilizer collides into the a node, the node is blown up to an isotrivial component with the same stabilizers at the two nodes and carrying the marking. 
    \item The total degree of the coarse map stays constant.
\end{enumerate}

In particular, condition $(1)$ tells us that any time two surfaces are attached along fibers, they must either be attached along nodal fibers, or in pairs consisting of $\mathrm{I}_n^*/\mathrm{I}_m^*/\mathrm{N}_1$ fibers, or in pairs $\mathrm{II}/\mathrm{II}^*, \mathrm{III}/\mathrm{III}^*$ and $\mathrm{IV}/\mathrm{IV}^*$. Furthermore, condition $(4)$ implies that the total number of nodal marked fibers in the degeneration of a marked rational elliptic surface must be $12$ (counted with multiplicity). 

\begin{remark} Deopurkar introduced the space $\mathrm{BrCov}_d(\overline{\calM}_{1,1}, b)$ of branched covers of a stacky curve in \cite{deopurkar}. He proves this is a smooth Deligne-Mumford stack that admits a birational morphism to the space of twisted stable maps. The morphism $\mathrm{BrCov}_d(\overline{\calM}_{1,1}, b) \to \calK(\mathscr{X}, d)$ is given by blowing up the isotrivial locus. As a result, the space $\mathrm{BrCov}_d(\overline{\calM}_{1,1},b)$ furnishes a smooth compactification of the space of elliptic surfaces. Later in Section \ref{sec:dp} we will study smoothness properties of our compactification $\calR(1/12 + \epsilon)$ that result from a combination of twisted stable maps and the minimal model program. \end{remark}

\section{Moduli spaces of weighted stable elliptic surface pairs}\label{sec:ksba}
In this section we recall the construction of KSBA compactifications $\calE_\calA$ of the moduli space of elliptic surfaces with section and $\calA$-weighted marked fibers. We follow the results of \cite{calculations}, \cite{tsm}, and \cite{master}. 

\subsection{Preliminaries from the MMP}

First, we recall the definition of a stable pair in the sense of the MMP. Let $X$ be a reduced projective variety and $D \subset X$ a $\bQ$-divisor.

\begin{definition} Let $X$ be a normal variety so that $K_X + D$ is $\bQ$-Cartier, and suppose there is a log resolution $f: Y \to X$ such that $$K_Y + \sum a_E E = f^*(K_X + D),$$ where the sum goes over all irreducible divisors on $Y$. Then the pair $(X,D)$ is \textbf{log canonical} (resp. \textbf{log terminal}) if all $a_E \leq 1$ (resp. $ < 1$). \end{definition}

\begin{definition}\label{def:slc} The pair $(X,D)$ has \textbf{semi-log canonical} (resp. \textbf{semi-log terminal}) singularities (or is an \textbf{slc} [resp. \textbf{slt}] pair) if:
\begin{itemize}
\item The $\bQ$-divisor $K_X + D$ is $\bQ$-cartier, 
\item the variety $X$ is $S2$,
\item $X$ has only double normal crossings in codimension 1, and
\item if $\nu: X^{\nu} \to X$ is the normalization, then the pair $(X^{\nu}, \nu_*^{-1}D + D^{\nu})$ is log canonical (resp. log terminal), where $D^{\nu}$ denotes the preimage of the double locus on $X^{\nu}$. 
\end{itemize}
\end{definition}

\begin{definition} A pair $(X,D)$ as above is a \textbf{stable pair} if
\begin{itemize}
\item $(X,D)$ is slc, and
\item $\omega_X(D)$ is ample.
\end{itemize}
\end{definition}

\subsection{Weighted stable elliptic surface pairs}
In \cite{master}, we define KSBA compcatifcations (c.f. \cite{ksb} and \cite{kp}) $\calE_\calA$ compactifying the moduli space of log canonical models $(f: X \to C, S + F_\calA)$ of $\calA$-weighted Weierstrass elliptic surface pairs (see Definition \ref{def:weierstrassmodel})  by allowing our surface pairs to degenerate to semi-log canonical (slc) pairs following the log minimal model program. For each \textit{admissible weight vector} $\calA$, we obtain a compactification $\calE_\calA$. These spaces parameterize slc pairs $(f: X \to C, S + F_{\calA})$, where $(f: X \to C, S)$ is an slc elliptic surface with section, and $F_{\calA} = \sum a_i F_i$ is a weighted sum of marked fibers with $\calA = (a_1, \dots, a_n)$ and $0 < a_i \leq 1$. 
  Our goal is then to compare moduli spaces for elliptic surfaces whose fibers have various weights $\calA$.

\smallskip

In particular, we prove the following theorems.

\begin{theorem}\cite[Theorem 1.1 \& 1.2]{master}\label{thm:stack} For admissible weights $\calA$ there exists a moduli pseudofunctor whose main component $\calE_{v, \calA}$ is representable by a proper Deligne-Mumford stack of finite type. The boundary of $\calE_{v, \calA}$ parametrizes $\calA$-broken elliptic surfaces (see Theorem \ref{thm:masterthm}). \end{theorem} 

\begin{remark} As the correct deformation theory for moduli of stable pairs has not yet been settled, we work with the normalization of the moduli stack (see Remark 3.7 and the discussion following in \cite{master} for more details). \end{remark}

In fact, the proof of properness is explicit -- we give a concrete stable reduction algorithm that allows us to determine what surfaces the boundary of our moduli space parametrizes. To reiterate, we state the following theorem.

\begin{theorem}\cite[Theorem 1.6]{master}\label{thm:masterthm} The boundary of the proper moduli space $\calE_{v,\calA}$ parametrizes
$\calA$-broken stable elliptic surfaces, which are pairs $(f: X \to C, S + F_{\calA})$
consisting of a stable pair $(X, S + F_\calA)$ with a map to a nodal curve $C$ such that:
\begin{itemize}
\item $X$ is an slc union of elliptic surfaces with section S and marked fibers, as well as \item chains of pseudoelliptic surfaces of type I and II (Definitions \ref{def:pseudo}, \ref{def:pseudotypeII}, and \ref{def:pseudotypeI}) contracted
by $f$ with marked pseudofibers. \end{itemize} \end{theorem}

\begin{figure}[!h]
\includegraphics{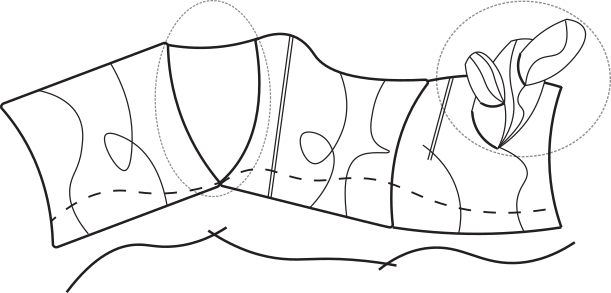}
\caption{An $\calA$-broken elliptic surface. Two types of pseudoelliptic surfaces circled. Left: Type II and Right: Type I.}\label{fig:pseudodef}
\end{figure}

To both put these results into context and define \emph{pseudoelliptic surfaces}, we must first discuss the different (singular) fiber types that appear in semi-log canonical models of elliptic fibrations as studied in \cite{calculations} (see also \cite[Section 3]{master}). 

\begin{definition}\label{def:fibertypes} Let $(g: Y \to C, S' + aF')$ be a Weierstrass elliptic surface pair over the spectrum of a DVR and let $(f: X \to C, S + F_a)$ be its relative log canonical model. We say that $X$ has a(n):  
\begin{enumerate} 
\item \textbf{twisted fiber} if the special fiber $f^*(s)$ is irreducible and $(X,S + E)$ has (semi-)log canonical singularities where $E = f^*(s)^{red}$;   
\item \textbf{intermediate fiber} if $f^*(s)$ is a nodal union of an arithmetic genus zero component $A$, and a possibly non-reduced arithmetic genus one component supported on a curve $E$ such that the section meets $A$ along the smooth locus of $f^*(s)$ and the pair $(X, S + A + E)$ has (semi-)log canonical singularities. 
\end{enumerate} 
\end{definition}

Given an elliptic surface $f: X \to C$ over the spectrum of a DVR such that $X$ has an \emph{intermediate fiber}, we obtain the \emph{Weierstrass model} (Definition \ref{def:weierstrassmodel}) of $X$ by contracting the component $E$, and we obtain the \emph{twisted model} by contracting the component $A$. As such, the intermediate fiber can be seen to interpolate between the Weierstrass and twisted models (see \ref{fig:transitions}). 

\begin{figure}[!h]
\includegraphics{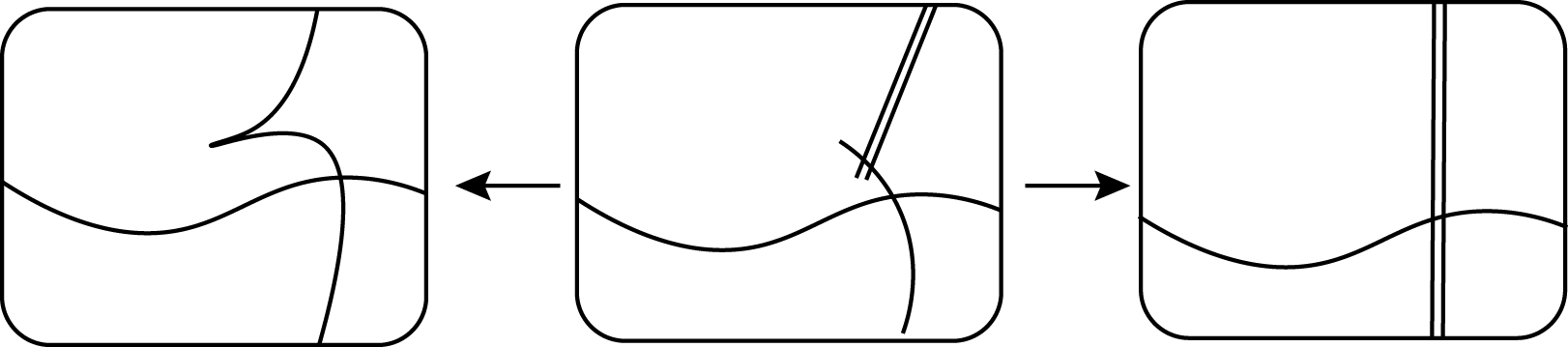}
\caption{Here we illustrate the relative log canonical models and morphisms between them. From left to right: \textbf{Weierstrass} model ($ 0 \leq a \leq a_0$) -- a single reduced and irreducible component meeting the section, \textbf{intermediate} model ($a_0 < a < 1$) -- a nodal union of a reduced component meeting the section and a nonreduced component, and \textbf{twisted} model ($a = 1$) -- a single nonreduced component meeting the section in a singular point of the surface.}\label{fig:transitions}
\end{figure}

This is made precise via the following theorem (see also \cite[Remark 3.20]{master}). 

\begin{theorem}\cite[Theorem 3.19]{master}\label{thm:transitions} Let $(g: Y \to C, S' + aF')$ be a Weierstrass model over the spectrum of a DVR, and let $(f: X \to C, S + F_a)$ be the relative log canonical model. Suppose the special fiber $F'$ of $g$ is either either 
\begin{enumerate*}[label = (\alph*)] \item one of the Kodaira singular fiber types, or \item $g$ is isotrivial with constant $j$-invariant $\infty$ and $F'$ is an $N_0$ or $N_1$ fiber (see Definition \ref{def:nk}). \end{enumerate*}

\begin{enumerate} 
\item If $F$ is a type $\mathrm{I}_n$ or $\mathrm{N}_0$ fiber, then the relative log canonical model is the Weierstrass model for all $0 \le a \le 1$.
\item For any other fiber type, there is an $a_0$ such that the relative log canonical model is
\begin{enumerate}[label = (\roman*)]
\item the \emph{Weierstrass} model for any $0 \le a \le a_0$,
\item a \emph{twisted fiber} consisting of a single non-reduced component supported on a smooth rational curve when $a = 1$, and
\item an \emph{intermediate fiber} with $E$ a smooth rational curve for any $a_0 < a < 1$.
\end{enumerate}
\end{enumerate}

The constant $a_0$ is as follows for the other fiber types:
$$
a_0 = \left\{ \begin{array}{lr} 5/6 & \mathrm{II} \\ 3/4 & \mathrm{III} \\ 2/3 & \mathrm{IV} \\  1/2 & \mathrm{N}_1 \end{array}\right. \\
\  a_0 = \left\{ \begin{array}{lr} 1/6 & \mathrm{II}^* \\ 
1/4  & \mathrm{III}^* \\
1/3 & \mathrm{IV}^* \\
1/2 & \mathrm{I}_n^* \end{array}\right. 
$$
\end{theorem}

\begin{remark} Note that if $F = E$ is a Weierstrass model of a type $\textrm{I}_n$ or $\textrm{N}_0$ fiber, then blowing up the point where $E$ meets the section produces an intermediate fiber $A \cup E'$ where $A$ is the exceptional divisor and $E'$ is the strict transform of $E$. While such intermediate fibers are never the log canonical models of pairs as above, they can appear in a stable degeneration where a \emph{pseudoelliptic surface} (see Definition \ref{def:pseudo}) is attached to $E$.  \end{remark} 

 We are now ready to recall the pseudoelliptic pairs that necessarily appear as components of surfaces in our moduli spaces, a phenomenon first noticed by La Nave \cite{ln}.

\begin{definition}\label{def:pseudo} A \textbf{pseudoelliptic pair} is a surface pair $(Z, F)$ obtained by contracting the section of an irreducible elliptic surface pair $(f: X \to C, S + F')$. We call $F$ the \textbf{marked pseudofibers} of $Z$. We call $(f: X \to C, S)$ the associated elliptic surface to $(Z, F)$. \end{definition}

These surfaces naturally appear as byproducts of the MMP. Indeed, the MMP will contract the section of an elliptic surface if it has non-positive intersection with the log canonical divisor of the surface. Furthermore, we note that there are two types of pseudoelliptic surfaces which appear on the boundary of our moduli spaces. We refer the reader to \cite[Definition 4.6, 4.7]{master} for the precise definitions of the two types of pseudoelliptic surfaces. We give abridged versions of the two definitions and refer to Figure \ref{fig:pseudodef} for brevity.

\begin{definition}\label{def:pseudotypeII} A pseudoelliptic surface of \textbf{Type II}  is formed by the log
canonical contraction of a section of an elliptic component attached along \emph{twisted} or \emph{stable} fibers. \end{definition}

\begin{definition}\label{def:pseudotypeI} A pseudoelliptic surface of \textbf{Type I} appear in \emph{pseudoelliptic trees} attached by gluing an
irreducible pseudofiber $G_0$ on the root component to an arithmetic genus one component $E$ of an
intermediate (pseudo)fiber of an elliptic or pseudoelliptic component. \end{definition}

See Remark \ref{rmk:flipping} for a further discussion on Type I pseudoelliptic surfaces. 

\begin{remark}\label{rmk:pseudofiberwt}  We recall the following from \cite[Definition 4.6]{master}. Let $(f: X' \to C, S + F_\calA)$ be an $\calA$-broken elliptic surface where $X' = X \cup_E Y$ with a marked Type I pseudoelliptic surface glued $(Y, (F_\calA)|_Y)$ glued to the arithmetic genus one component $E$ of an intermediate (pseudo)fiber $E \cup A$ with reduced component $A$ on $X$.  Then if $F_{\calA} = \sum a_i F_i$ we have that \begin{equation}\label{eq:coeff}\mathrm{Coeff}(A, F_\calA) = \sum_{\mathrm{Supp}(F_\calA|_Y)} \mathrm{Coeff}(F_i) = \sum_{\mathrm{Supp}(F_\calA|_Y)} a_i.\end{equation} is a sum of weights of marked fibers on $Y$. \end{remark}

To reiterate, Figure \ref{fig:pseudodef} has a Type II pseudoelliptic component circled on the left, and a tree of pseudoelliptic surfaces of Type I circled on the right. Furthermore, it turns out that contracting
the section of a component to form a pseudoelliptic corresponds to stabilizing the base curve
as an $\calA$-stable curve in the sense of Hassett (see \cite[Corollaries 6.7 \& 6.8]{calculations}). In particular we have the following theorem.

\begin{theorem}\cite[Theorem 1.4]{master}\label{thm:forgetful} There are forgetful morphisms $\calE_{v, \calA} \to \overline{\calM}_{g, \calA}$. \end{theorem}

\begin{remark}\label{rmk:contraction} We recall that for an irreducible component with base curve $\bP^1$ and $\deg \mathscr{L} = 1$,  contracting the section of an elliptic component might \emph{not} be the final step in the minimal model program. In particular, we might need to contract the entire pseudoelliptic component to a curve or a point. This is the content of \cite[Proposition 7.4]{calculations}. \end{remark}

\subsection{Wall and chamber structure}

We now want to understand how the moduli spaces $\calE_{\calA}$ change as we vary the weight vector $\calA$. 

\smallskip

Let $\calD \subset (\bQ \cap [0,1])^n$ be the set of \emph{admissible weights},
i.e. weight vectors $\calA$ such that $K_X +S+F_\calA$ is pseudoeffective. A \emph{wall and chamber decomposition}
of $\calD$ is a finite collection $\calW$ of hypersurfaces (the walls), and the chambers are the connected
components of the complement of $\calW$ in $\calD$. There are three types of walls in our wall and chamber decomposition.

\begin{definition}\label{def:walltypes} The collection $\calW$ consists of the following types of walls:
\begin{enumerate}
\item[(I)] A wall of \textbf{Type $\mathrm{W}_{\mathrm{I}}$} is a wall arising from the log canonical transformations seen in
Theorem \ref{thm:transitions}, i.e. the walls where the fibers of the relative log canonical model
transition from twisted, to intermediate, to Weierstrass fibers.
\item[(II)] A wall of \textbf{Type $\mathrm{W}_{\mathrm{II}}$} is a wall at which the morphism induced by the log canonical contracts the section of some components.
\item[(III)] A wall of \textbf{Type $\mathrm{W}_{\mathrm{III}}$} is a wall where the morphism induced by the log canonical contracts an entire rational pseudoelliptic component (Remark \ref{rmk:contraction}). \end{enumerate} \end{definition}

Note that there are also \emph{boundary} walls given by $a_i = 0$ and $a_i = 1$ at the boundary of $\calD$, and these can be any of the three types above. 

\begin{remark}\label{rmk:hassettwall2} Recall that by the discussion proceeding Theorem \ref{thm:forgetful}, that the walls of Type $\mathrm{W}_{\mathrm{II}}$ are precisely the walls of Hassett's wall and chamber decomposition \cite{has}. \end{remark}

The following theorem classifies the three types of walls in general.

\begin{theorem}\cite[Theorem 6.3]{master}\label{thm:walls} The non-boundary walls of each type are described as follows: 
\begin{enumerate}[label = (\alph*)]
\item Type $\mathrm{W_{\mathrm{I}}}$ walls are defined by the equations
$$
a_i =\frac{1}{6}, \frac{1}{4}, \frac{1}{3}, \frac{1}{2}, \frac{2}{3}, \frac{3}{4}, \frac{5}{6}.
$$ 

\item Type $\mathrm{W_{\mathrm{II}}}$ walls are defined by equations
$$
\sum_{j = 1}^k a_{i_j} = 1.
$$
where $\{i_1, \ldots, i_k\} \subset \{1, \ldots, n\}$. When the base curve is rational there is another $\mathrm{W_{\mathrm{II}}}$ wall at
$$
\sum_{i = 1}^r a_i = 2.
$$
\item Type $\mathrm{W_{\mathrm{III}}}$ walls where a rational pseudoelliptic component contracts to a point are given by
$$
\sum_{j = 1}^k a_i = c
$$
where $\{i_1, \ldots, i_k\} \subset \{1, \ldots, n\}$ and $c = \frac{1}{6}, \frac{1}{4}, \frac{1}{3}, \frac{1}{2}, \frac{2}{3}, \frac{3}{4}, \frac{5}{6}$ are the log canonical thresholds of minimal Weierstrass fibers. \\

\item Finitely many Type $\mathrm{W_{\mathrm{III}}}$ walls where an isotrivial rational pseudoelliptic component contracts onto the $E$ component of a pseudoelliptic surface it is attached to. 

\end{enumerate}

\noindent In particular, there are only finitely many walls and chambers. 
\end{theorem}

One of the first goals of this paper is to explicitly determine the walls of Type $\mathrm{W_{\mathrm{III}}}$ for $\calE_{\calA}$ in the case of rational elliptic surfaces. 

\smallskip

Finally, we state one of the main results from \cite{master}, which states how the moduli space changes as we vary $\calA$. 

\begin{theorem}\cite[Theorem 1.5]{master}\label{thm:main} Let $\calA, \calB \in \bQ^r$ be weight vectors with $0 < \calA \leq \calB \leq 1$. Then
\begin{enumerate}
\item If $\calA$ and $\calB$ are in the same chamber, then the moduli spaces and universal families are isomorphic.
\item If $\calA \le \calB$ then there are reduction morphisms $\calE_{v,\calB} \to \calE_{v,\calA}$ on moduli spaces which are compatible with the reduction morphisms on the Hassett spaces:
$$
\xymatrix{\calE_{v,\calB} \ar[r] \ar[d] & \calE_{v,\calA} \ar[d] \\ \overline{\calM}_{g,\calB} \ar[r] & \overline{\calM}_{g,\calA}}
$$

\item The universal families are related by a sequence of explicit divisorial contractions and flips $\calU_{v,\calB} \dashrightarrow \calU_{v,\calA}$ such that the following diagram commutes:
$$
\xymatrix{\calU_{v,\calB} \ar@{-->}[r] \ar[d] & \calU_{v,\calA} \ar[d] \\ \calE_{v,\calB} \ar[r] & \calE_{v,\calA}}
$$
More precisely, across $\wi$ and $\wiii$ walls there is a divisorial contraction of the universal family and across a $\wii$ wall the universal family undergoes a log flip. 
\end{enumerate}\end{theorem}

\begin{remark}\label{rmk:flipping} For more on Theorem \ref{thm:main} (3), we refer the reader to \cite[Section 8]{master}. La Nave (see \cite[Section 4.3, Theorem 7.1.2]{ln}) noticed that the contraction of the section of a component is a log flipping contraction inside the total space of a one parameter degeneration.  In particular, the Type I pseudoelliptic surfaces are thus attached along the reduced component of an intermediate (pseudo)fiber (see \cite[Figure 13]{master}). \end{remark}

\subsection{The moduli spaces of interest} 

From now on we restrict to the case of rational elliptic surfaces. In particular, the base curve is of genus zero and the degree of $\mathscr{L}$ is one so that $C \cong \mb{P}^1$ and $\mathscr{L} = \calO_{\mathbb{P}^1}(1)$. Let $\mathcal{E}_{1,\calA}$ denote the KSBA compactification of the stack of rational elliptic surfaces with $12$ marked fibers weighted by $\calA = (a_1, \ldots, a_{12})$. Note here that the choice of $12$ marked divisors is part of the data. Thus generically $\mathcal{E}_{1,\calA}$ is fibered over the stack of rational elliptic surfaces with section and the fibers are open subvarieties of $\mb{P}(|-K_X|)^{12}$. 

First we pick out a slice of this projection. 

\begin{definition}\label{def:es} Let $\mathcal{E}^s_{1,\calA}$ be the closure in $\mathcal{E}_{1,\calA}$ of the locus of pairs $(f : X \to C, S +F_\calA)$ where $X$ is a rational elliptic surface and $\mathrm{Supp}(F_\calA)$ consists of $12$ $\mathrm{I}_1$ singular fibers. \end{definition}

Equivalently, $\mathcal{E}^s_{1,\calA}$ is the closure of the locus of rational elliptic surface pairs with smooth log canonical model and with divisor given by the discriminant of the fibration. Note that when $\calA = (a,\ldots, a)$ is a constant weight vector, then $S_{12}$ acts on $\mathcal{E}^s_{1,\calA}$ by permuting the marked fibers.

\begin{definition}\label{def:Ra} For $\calA = (a, \ldots, a)$ the constant weight vector, we define
$
\mathcal{R}(a) := \mathcal{E}^s_{a,\ldots,a}/S_{12}.$ \end{definition}

The following is clear from the analagous statement for $\calE_{1,\calA}$: 

\begin{prop} $\calR(a)$ is a proper Deligne-Mumford stack with coarse moduli space $R(a)$. \end{prop}

The stack $\mathcal{R}(a)$ and its coarse moduli space are the main subjects of this paper but $\mathcal{E}^s_{1,\calA}$ and $\mathcal{E}_{1,\calA}$ will be used in the sequel to understand the wall and chamber structure and birational transformations of $\mathcal{R}(a)$.

\begin{remark}\label{rem:disc} Since $\calE^s_{1,\calA}$ is defined as the closure of the locus of rational elliptic surface pairs with the $12$ $\mathrm{I}_1$ fibers marked, then the marked fibers always appear at the discriminant $\mathscr{D}$ of $f : X \to C$ over the smooth locus $C^{sm}$. In particular, the map $\calE^s_{1,\calA} \to \overline{\calM}_{0,\calA}$ sends $(f : X \to C, S + F_\calA)$ to $C$ marked by the $\calA$-weighted discriminant $\mathscr{D}$ of $f$. 

\end{remark}

\section{Birational contractions of the moduli space across walls}\label{sec:walls}

We saw above (see Theorem \ref{thm:walls}) that we have a complete description of the location of walls of Types $\wi$ and $\wii$. The goal of this section is to describe the walls of type $\wiii$ for the moduli spaces $\calE_{\calA}$ and then use this to study the explicit birational contractions $\mathcal{R}(a)$ undergoes as one reduces $a$. 

\subsection{Pseudoelliptic contractions of $\calE_\calA$}

First, recall that walls of Type $\wiii$ (see Definition \ref{def:walltypes}) correspond to the contraction of an entire pseudoelliptic component. From Definition \ref{def:pseudotypeI}, we noted that pseudoelliptic components of Type I are connected to the arithmetic genus one component $E$ of an intermediate (pseudo)fiber of another component. 

Let $(f : X \cup Z \to C, S + F_\calA)$ be an $\calA$-broken elliptic surface with pseudoelliptic component $Z$ attached to the arithmetic genus one component $E$ of an intermediate (pseudo)fiber $A \cup E$ on $X$. Suppose further that $Z$ is rational, otherwise $Z$ never contracts with nonzero coefficients (see \cite[Corollary 6.10]{calculations}). Then the contraction of $Z$ to a point produces a minimal Weierstrass fiber at $A \cup E$. Furthermore, $Z$ contracts if and only if $E$ is contracted in the log canonical model of $(X, (S + F_\calA)|_X)$ (see \cite[Proposition 7.4]{calculations}). 

\begin{definition} Let $(X,D)$ be be a pair with (semi-)log canonical singularities and $A \subset X$ a divisor. The \textbf{(semi-)log canonical threshold $\mathrm{lct}(X,D,A)$} is 
$$
\mathrm{lct}(X,D,A) := \max\{a \ : \ (X, D + aA) \text{ has (semi-)log canonical singularities }\}.
$$
\end{definition}

 Let $(f : X \cup Z \to C, S + F_\calA)$ be as above and let $p: X \to X'$ be the contraction of the $A \cup E$ intermediate fiber onto its Weierstrass (pseudo)fiber $A' \subset X'$. Let $D' = f_*D \subset X'$ where
 $$
 D = (S + F_\calA)|_X - \mathrm{Coeff}(F_\calA,A)A
 $$
 the boundary divisor on $X$ excluding the component $A$. 

\begin{prop}\label{prop:typeIII} The component $Z$ contracts to a point in the log canonical model of\\ $(f : X \cup Z \to C, S + F_\calA)$ if and only if
$$
\sum_{\mathrm{Supp}(F_\calA|_Z)} a_i \le \mathrm{lct}(X',D', A')$$
where the left hand side is a sum over marked pseudofibers on $Z$. 
\end{prop}

\begin{cor}\label{cor:typeIII} There are type $\wiii$ walls for $\calE_\calA$ corresponding to pseudoelliptic components contracting to a point given by
$$
\sum_{i \in I} a_i = c
$$
where $I \subset \{1, \ldots, n\}$ and $c$ is the log canonical threshold of a minimal Weierstrass cusp. 
\end{cor} 

\begin{remark} In the case of type $\mathrm{II}$, $\mathrm{III}$, $\mathrm{IV}$ and $\mathrm{N}_1$ Weierstrass cusps, the log canonical threshold $c$ is given by the numbers $a_0$ in Theorem \ref{thm:transitions}. 
\end{remark} 

\begin{remark} Note there are also type $\wiii$ contractions of pseudoelliptics at the boundary walls given by $a_i = 0$. \end{remark} 

\begin{proof} The component $Z$ contracts to a point if and only if the curve $E$ it is attached to contracts to a point in the log canonical model of $(X, (S + F_\calA)|_X)$. Note first that $\mathrm{Coeff}(E, (S + F_\calA)|_X) = 1$ since $E$ is in the double locus of $X \cup Z$ and
$$
\sum_{\mathrm{Supp}(F_\calA|_Z)} a_i = \mathrm{Coeff}(A, (S + F_\calA)|_X)
$$
by Equation \ref{eq:coeff}. 

We need to compute at which coefficient of $A$ the component $E$ is contracted in the log canonical model of $(X, (S + F_\calA)_X)$. Since this is a local question, we may assume $X$ is an elliptic surface with section $S$, intermediate fiber $A \cup E$ and Weierstrass model $p : X \to X'$ with Weierstrass cusp $p_*(A \cup E) = A'$. 

Suppose that $a \le \mathrm{lct}(X', S', A')$ where $S' = p_*S$. Consider the log resolution $p : X \to X'$ of the pair $(X', S' + aA')$. Then by definition of log canonical singularities, the log canonical model of $(X, S + aA + E)$ relative to $p$ is $X'$ since $E = \mathrm{Exc}(p)$. Conversely, it is easy to compute from the singularity of $X$ at $A \cap E$ that $\mathrm{lct}(X, S + E, A) = 1$ (see the computations in \cite{calculations}). If $1 > a > \mathrm{lct}(X', S', A')$, the pair $(X, S + aA + E)$ is log canonical while the contraction of $E$ produces pair that has worse than log canonical singularities and so $E$ cannot be contracted in the lc model.  \end{proof}

\subsection{The birational contractions of $\calR(a)$}

Now we use the above discussion to determine the walls of $\calR(a)$ as one decreases $a$ and what birational contractions the moduli space undergoes. 

\begin{lemma}\label{lem:typeIRa} There are Type $\wii$ walls where Type $\mathrm{I}$ pseudoelliptic surfaces of $\calR(a)$ form at $a = 1/k$ for $k = 1, \ldots, 5$. 
\end{lemma}
\begin{proof} The flips forming Type I pseudoelliptic curves form when a component of the underlying weighted curve is contracted. Since all weights are the same, this occurs when $ka = 1$ as long as the total weight $12a > 2$ so that the moduli space of weighted stable curves is nontrivial. 
\end{proof} 

\begin{lemma}\label{lem:connection} Let $a > 1/6$. Then any pseudoelliptic component on a surface parametrized by $\calR(a)$ must be a Type $\mathrm{I}$ pseudoelliptic glued along a type $\mathrm{II}, \mathrm{III}, \mathrm{IV},$ or $\mathrm{N}_1$. \end{lemma}

\begin{proof} Let $Z$ be such a component. Then it is formed by a pseudoelliptic flip corresponding to a type $\wii$ wall as in Lemma \ref{lem:typeIRa}. In particular, the number of maked points on the section component that contracted to form $Z$ is at most $5$. Since the marked points occur at the discriminant of the elliptic fibration (counted with multiplicity) then $Z$ must be a component with at most $5$ singular fibers counted with multiplicity away from the double locus. 

Since $Z$ is a rational pseudoelliptic surface the total multiplicity of the discriminant of the corresponding elliptic surface is 12 (see Remark \ref{rmk:singfibers}). Therefore the pseudofiber of $Z$ where it is attached must correspond to a fiber with at discriminant at least $7$ so it has to be an $\mathrm{I}_n^*$ for $n > 0$, $\mathrm{II}^*$, $\mathrm{III}^*$ or $\mathrm{IV}^*$. 

In the first case, it must be attached to another $\mathrm{I}_m^*$ or an $\mathrm{N}_1$ fiber by the balancing condition in \ref{tsmconditions}. By degree considerations it has to be attached to an $\mathrm{N}_1$ fiber. In the latter case, the balancing condition requires it be attached to a type $\mathrm{II}, \mathrm{III}$ or $\mathrm{IV}$ respectively. \end{proof}

\begin{lemma}\label{lem:stablecurves} Let $\calA = (a, \ldots, a)$ for $a = 1/6 + \epsilon$. Then curves $C$ parametrized by $\overline{\calM}_{0, \calA}$ are either
\begin{enumerate}
\item a smooth $\bP^1$ with 12 marked points, or
\item the union of two rational curves, each with 6 marked points.
\end{enumerate}
 \end{lemma} 

\begin{proof} It is clear that $C$ can be a smooth $\bP^1$. If $C$ is the union of two rational components, then since each point is weighted by $1/6 + \epsilon$, and since each curve has to have total weight $>2$ including the node, each curve must have six points. Suppose $C = \cup_{i =1}^3 C_i$, and label the two end  components by $C_1$ and $C_3$, and the bridge by $C_2$. Then at least one of $C_1$ and $C_3$ will not be stable as $5 \cdot (1/6 + \epsilon) < 1$. \end{proof}

\begin{cor}\label{cor:nocomponents} Let $X$ be a surface parametrized by $\calR(1/6 + \epsilon)$. Then $X$ has at most two elliptic components. \end{cor}

\begin{remark} $X$ can have many Type I pseudoelliptic components mapping onto marked points of $C$. \end{remark}

\begin{definition} If $X$ parametrized by $\calR(1/6 + \epsilon)$ has a single (resp. exactly two) fibered component(s) $X_0$ (resp. $X_0\cup X_1$), we call $X_0$ (resp. $X_0 \cup X_1$) the \textbf{main component} of $X$. 
\end{definition} 

Note in particular that every surface parametrized by $\calR(1/6 + \epsilon)$ consist of a main component with trees of pseudoelliptics attached along Type $\mathrm{II}, \mathrm{III}, \mathrm{IV}$ and $\mathrm{N}_1$ fibers. 

\begin{prop}\label{prop:typeii6} There is a wall at $a = 1/6$ where the entire section contracts and the Hassett moduli space becomes a point. Furthermore
\begin{enumerate}
\item  If $X$ has an irreducible main component $X_0$ then $X_0$ contracts to a degree one del Pezzo surface with trees of pseudoelliptics branching off.
\item If $X$ has main component $X_0 \cup X_1$, then it either contracts to the above case or it contracts to a union of Type $II$ pseudoelliptics $Y_0 \cup Y_1$ glued along a twisted pseudofibers with trees of pseudoelliptics branching off. \end{enumerate}

Furthermore, in the latter case $Y_0 \cup Y_1$ are glued along twisted $\mathrm{I}_0^*/\mathrm{I}_0^*, \mathrm{I}_0^*/\mathrm{N}_1$ or $\mathrm{N}_1/\mathrm{N}_1$ pseudofibers. 
\end{prop}

\begin{proof} If the main component is irreducible, then every other component lies on a Type I pseudoelliptic tree glued along intermediate $\mathrm{II}, \mathrm{III}, \mathrm{IV}$ or $\mathrm{N}_1$ fibers of $X_0$ by Lemma \ref{lem:typeIRa}. Otherwise the fibered components are of the form $X_0 \cup_E X_1 \to C$ where $C = C_0 \cup_p C_1$ is the nodal union of two $6$-pointed rational curves by Lemma \ref{lem:stablecurves} and $E$ is a twisted fiber of both $X_0$ and $X_1$. If $X_i \to C_i$ is a normal elliptic fibration, then it must have $6$ singular fibers counted with multiplicity other than the double locus $E$. Thus $E$ must contribute $6$ to the discriminant and so is an $\mathrm{I}_0^*$. If $X_i$ is a non-normal component then it must be an isotrivial $j$-invariant $\infty$ component. 

If it is trivial then it contracts onto a nodal fiber of the other component producing a surface with a single main component. If it is nontrivial, then it must be unique $2\mathrm{N}_1$ surface with $\deg(\mathscr{L}) = 1$ with $6$ marked fibers counted with multiplicity and glued along a twisted fiber with $\mathbb{Z}/2\mathbb{Z}$ stabilizer. This means the two main components are attached along $\mathrm{N}_1/\mathrm{I}_n^*$ or $\mathrm{N}_1/\mathrm{N}_1$. Again by examining degrees of the discriminant we see in the former $n = 0$.   \end{proof}

\begin{cor}\label{cor:typeIR6} Let $1/12 < a \le 1/6$. Then the surfaces parametrized by $\calR(a)$ consist of the following two types:
\begin{enumerate}
    \item An irreducible pseudoelliptic main component with trees of Type I pseudoelliptics attached to it along $\mathrm{II}, \mathrm{III}, \mathrm{IV}$ or $\mathrm{N}_1$ pseudofibers,
    \item A main component consisting of two Type $\mathrm{II}$ pseudoelliptics glued along twisted $\mathrm{I}_0^*/\mathrm{I}_0^*$, $\mathrm{N}_1/\mathrm{I}_0^*$ or $\mathrm{N}_1/\mathrm{N}_1$ pseudofibers with Type $\mathrm{I}$ pseudoelliptic trees attached to it along $\mathrm{II}, \mathrm{III}, \mathrm{IV}$ or $\mathrm{N}_1$ pseudofibers.
\end{enumerate}
\end{cor}

\begin{theorem}\label{thm:type1walls} The Type $\wiii$ walls of $\calR(a)$ corresponding to the contraction of a Type $\mathrm{I}$ pseudoelliptic component to a point occur at $\{a = a_0/k\}$ for $2 \leq k \leq 5$ and $a_0$ is one of the four constants appearing in Theorem \ref{thm:transitions} for fibers of type $\mathrm{II}, \mathrm{III}, \mathrm{IV},$ and $\mathrm{N}_1$. \end{theorem}

\begin{proof} By Lemma \ref{lem:typeIRa} and Corollary \ref{cor:typeIR6}, any Type I pseudoelliptic consists of a surface with $2 \le k \le 5$ marked fibers (counted with multiplicity) and a $\mathrm{II}^*, \mathrm{III}^*, \mathrm{IV}^*$, $\mathrm{I}_n^*$ ($n > 0$), or $\mathrm{N}_1$ fiber attached to a type $\mathrm{II}, \mathrm{III}, \mathrm{IV}, \mathrm{N}_1$ or $\mathrm{N}_1$ respectively. By Corollary \ref{cor:typeIII} these surfaces contract when $ka = c$ for $c$ the log canonical threshold of type $\mathrm{II}, \mathrm{III}, \mathrm{IV}$ or $\mathrm{N}_1$ minimal Weierstrass cusp respectively. The log canonical thresholds are the $a_0$ in Theorem \ref{thm:transitions}. 
\end{proof}

\begin{theorem}\label{thm:jinfwalls}(see Example \ref{ex:in} and Figure \ref{fig:ex1}) There are walls of type $\mathrm{W}_{\mathrm{III}}$ at $a = 1/k$ for $2 \leq k \leq 9$. Where a trivial component of $j$-invariant infinity contracts onto its attaching (pseudo)fiber. \end{theorem}
\begin{proof} 
Trivial $j$-invariant infinity components appear when marked fibers collide and carry the number of markings that collide to form the component. If such a component $Z \subset X$ has $k \le 6$ marked fibers, then it must contract onto the fiber direction at the Type $\wii$ walls $a = 1/k$ where the corresponding section contracts to a point. 

Suppose $X = Z \cup_E Y$ with $Z$ carrying $k \geq 7$ marked fibers. Then at coefficients $a = 1/6 + \epsilon$, the surface $Z$ is the main component and $Y$ is a Type I pseudoelliptic tree. In particular the trivial component $Z$ is blown up at the point where the fiber $E$ meets the section. Then at $a = 1/6$ the section contracts and so the main component $Z$ becomes a nontrivial $\mb{P}^1$ bundle over the nodal curve $E$ and the marked pseudofibers become sections of the projection $Z \to E$ and the flipped curve $A$ in the intermediate pseudofiber $A \cup E$ becomes a fiber of this projection. Then one may compute that when $a = 1/k$ the restriction of the log canonical divisor to $Z$ is linearly equivalent to $A$ and so the the component $Z$ contracts along the projection $Z \to E$. 

Finally $k \le 9$ because $k \geq 10$ fibers on a rational elliptic surface cannot collide (see e.g. Persson's classification of singular fibers \cite{persson}). \end{proof}

\subsection{Examples}\label{sec:examples}

\begin{example}[See Figure \ref{fig:ex1}]\label{ex:in} Suppose $X_{\eta}$ is a smooth rational elliptic surface with 12 $(\mathrm{I}_1)$ fibers, and suppose it appears as the general fiber of a family $\mathscr{X} \to B$. We will compute the stable limit of this family when seven of the nodal fibers collide for all weights $a$. We will use $X^0_{a}$ to denote the special fiber of $\mathscr{X} \to B$. 

 We begin with twisted stable maps limit at $a = 1$. The surface $X^0_{1}$ is the union of two surfaces, $X^0_1 = Z \cup_{\mathrm{I}_7} Y$, where $Z$ is a trivial nodal elliptic surface $\bP^1 \times E$ and seven marked fibers, glued to $Y$ along an $\mathrm{I}_7$ fiber.

 At $a = 1/5$ the section of $Y$ contracts to obtain $X^0_{1/5} = Z \cup_{\mathrm{I}_7} \tilde{Y}$, where $\tilde{Y}$ is the pseudoelliptic surface corresponding to $Y$.  Decreasing weights so that $\calA = 1/5 - \epsilon$ we cross a wall of type $\wii$ and a flip occurs in the family to obtain $X^0_{1/5 - \epsilon} = \hat{Z} \cup_{\mathrm{I}'_7} \tilde{Y}$, where we blow up a point on $Z$ corresponding to the contraction of the section of $Y$. We note that $\mathrm{I}'_7$ is an intermediate fiber on $\hat{Z}$ and is the union of a genus one component $E$ with a genus zero component $A$.

At $a = 1/6$ the section of $\hat{Z}$ contracts to form $\tilde{Z}$. Since $\tilde{Z}$ is the blowdown of the strict transform of the section of the blowup of a trivial surface $\mb{P}^1 \times E$, then $\tilde{Z}$ is a $\mb{P}^1$ nontrivial $\mb{P}^1$-bundle over $E$. The component $A$ becomes a fiber of the projection $\tilde{Z} \to E$ and the marked pseudofibers become sections. At $a = 1/7$, the surface $\tilde{Z}$ contracts onto the $E$ component (i.e. $\mathrm{I}_7$ pseudofiber of $Y$) and we are left with $X^0_{1/7}$ which is a single pseudoelliptic component with an $\mathrm{I}_7$ pseudofiber and five $\mathrm{I}_1$ pseudofibers. \end{example}

\begin{figure}[ht]
\centering
\includegraphics[scale=.85]{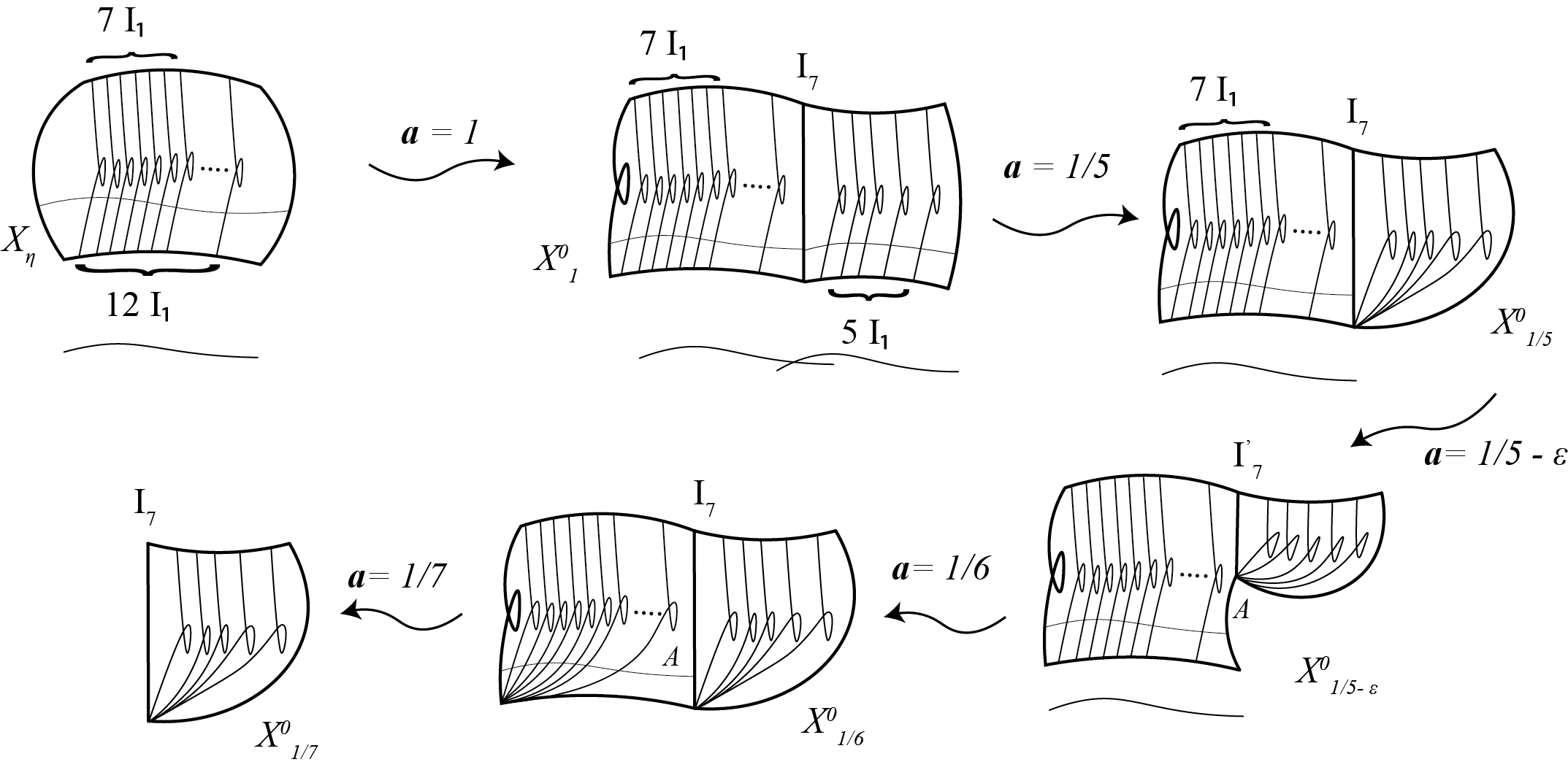}
\caption{Example \ref{ex:in} illustrating Theorem \ref{thm:jinfwalls}.}\label{fig:ex1}
\end{figure}

\begin{remark} The above example occurs when $6 \leq k \leq 9$ nodal fibers collide, and the final wall is at $1/k$. Note that for numerical reasons, we cannot have $10 \leq n \leq 12$ nodal fibers collide (see \cite{persson}). \end{remark}

\begin{example}[See Figure \ref{fig:ex2}]\label{ex:jinfinity} Suppose $X^0_1$ is the twisted stable maps limit of a family $\mathscr{X} \to B$ and $X^0_{1} = Y \cup_{\mathrm{I}_0^*} Z$, where there are six marked fibers on both irreducible components, and there is a $\mathbb{Z}/2\mathbb{Z}$ stabilizer at node of the stable base curve corresponding to the double locus of $X^0_1$. Suppose further that $j(\mathrm{I}_0^*) = \infty$. Then the marked $j$-invariant infinity fibers on $Y$ or $Z$ can collide into the double locus and we obtain new isotrivial components of $j$-invariant infinity.

If $1 \leq n \leq 5$ marked fibers on $Y$ collide onto the double locus, the stable limit will be a new surface $Y' \cup W \cup Z$, where each component has $(6-n)$, $n$ and $6$ marked fibers respectively, $Y'$ now has an $\mathrm{I}_n^*$ fiber where the markings collided, and the component $W$ is isotrivial of $j$-invariant infinity. By examining the stabalizer of the twisted stable map at the nodes, we see that $W$ has an $\mathrm{N}_1$ fiber glued to the $\mathrm{I}_0^*$ fiber of $Z$ and another $\mathrm{N}_1$ fiber glued to the $I_n^*$ fiber of $Y'$. When the the coefficients decrease past a Type $\wii$ wall $a = 1/(6-n)$, $Y'$ undergoes a pseudoelliptic flip explaining how isotrivial main components appear in Proposition \ref{prop:typeii6}. 

\end{example}

\begin{figure}[ht]
\centering
\includegraphics{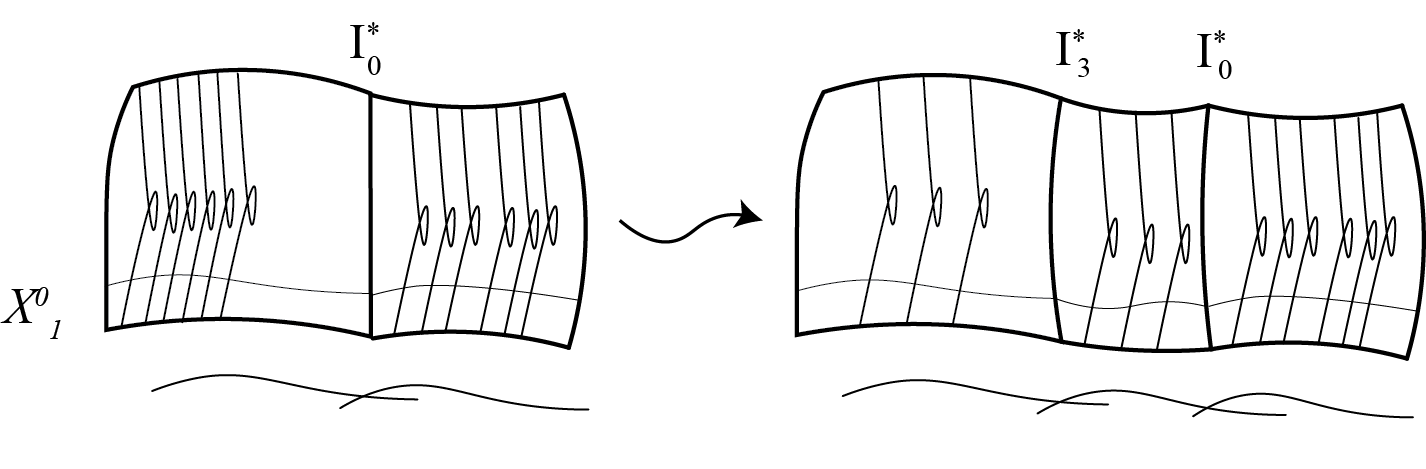}
\caption{Example \ref{ex:jinfinity} with $3$ $\mathrm{I}_1$ fibers colliding into the $\mathrm{I}_0^*/\mathrm{I}_0^*$ double locus.}\label{fig:ex2}
\end{figure}

\begin{remark}\label{rmk:jinfinity} We note that in the above example there can be  a chain of isotrivial $j$-invariant infinity surfaces sandwiched between the two non-isotrivial surfaces if $\mathrm{I}_n$ type fibers collide into the double locus from multiple sides. However, each end surface must have at least $2$ singular marked fibers (counted with multiplicity) and each isotrivial surface must have at least $1$ and so the maximum length of such a chain is $8$.  \end{remark}

\section{Background on $\Q$-Gorenstein Deformations}\label{sec:hacking}
We follow Hacking \cite[Section 3]{hacking}. Let $p \in X$ be the germ of an slc surface. The canonical covering $\pi: Z \to X$ is defined by 
$$Z := \underline{\Spec}_X (\calO_X \oplus \calO_X(K_X) \oplus \dots \oplus \calO_X((N-1)K_X)),$$
where $N$ is the index of $P \in X$ and the multiplication is given by choosing an isomorphism $\calO_X(NK_X) \to \calO_X$. The morphism $\pi: Z \to X$ is a cyclic quotient of degree $N$, and the surface $Z$ is Gorenstein.

\begin{definition} Let $p \in X$ be an slc surface germ. Let $N$ be the index of $X$ and let $Z \to X$ be the canonical covering (a $\mu_N$ quotient). A deformation $\mathcal X / (0 \in S)$ of $X$ is \textbf{$\bQ$-Gorenstein} if there is a $\mu_N$-equivariant deformation $\mathcal{Z}/S$ of $Z$ whose quotient is $\mathcal X / S$. \end{definition}

For such $X$, the canonical covering at a point $p \in X$ is uniquely determined in the \'etale topology, and therefore this defines a Deligne-Mumford stack $\mathscr{X}$ with coarse moduli space $X$, the canonical covering stack of $X$. 

\begin{lemma}\cite[Lemma 3.5]{hacking} Let $p \in X$ be an slc surface germ of index $N$ and $Z \to X$ the canonical covering with group $\mu_N$. Let $\mathcal Z / (0 \in S)$ be a $\mu_N$-equivariant deformation of $Z$ inducing a $\bQ$-Gorenstein deformation $\mathcal X / (0 \in S)$ of $X$. There is an isomorphism $$\mathcal Z \cong \underline{\Spec}_{\mathcal X} (\calO_{\mathcal X} \oplus \omega_{\mathcal X / S} \oplus \dots \oplus \omega_{\mathcal X / S}^{[N-1]}),$$ with multiplication given by a fixing a trivialization of $\omega_{\mathcal X / S}^{[N]}$. In particular, $\mathcal Z / S$ is determined by $\mathcal X / S$. \end{lemma}

Let $A$ be $\mathbb{C}$-algebra and $M$ a finite $A$-module. For $\mathcal X / A$ a flat family of schemes over $A$, let $L_{\mathcal X / A}$ denote the cotangent complex of $\mathcal X / A$. Define 
\begin{align*} T^i(\mathcal X / A, M) &= \mathrm{Ext}^i(L_{\mathcal X / A}, \calO_{\mathcal X} \otimes_A M)\\
\mathscr T^i(\mathcal X/A, M) &= \mathcal E xt^i(L_{\mathcal X/A}, \calO_{\mathcal X} \otimes_A M)\end{align*}

For $\mathcal X / A$ a $\bQ$-Gorenstein family of slc surfaces over $A$, let $\mathscr{X} / A$ denote the canonical covering stack of $\mathcal X / A$, and let $p: \mathscr{X} \to \mathcal X$ be the induced map. Define \begin{align*} T^i_{QG} (\mathcal X / A, M) &= \mathrm{Ext}^i(L_{\mathscr{X} / A}, \calO_{\mathscr{X}} \otimes_A M)\\  \mathscr T^i_{QG} (\mathcal X / A, M) &= p_*\mathcal E xt^i (L_{\mathscr{X} / A}, \calO_{\mathscr{X}} \otimes_A M). \end{align*}
For $X/\mathbb{C}$ define $T^i_X,\mathscr T^i_X, T^i_{QG,X}, \mathscr T^i_{QG,X}$ by $T^i_X:= T^i(X/\mathbb{C}, \mathbb{C})$ and so on. 

If $A' \to A$ is an infinitesimal extension of $A'$, then isomorphism classes of $\bQ$-Gorenstein deformations of $\mathcal X / A$ over $A'$ are in bijection with the set of isomorphism classes of deformations of $\mathscr{X} / A$ over $A'$ \cite[Proposition 3.7]{hacking}.

Furthermore, letting $A = M = \mathbb{C}$, we see that first order $\bQ$-Gorenstein deformations of $X/\mathbb{C}$ are identified with $T^1_{QG, X}$ and the obstructions to extending $\bQ$-Gorenstein deformations lie in $T^2_{QG, X}$. Furthermore, $\mathscr T^0_{QG, X} = \mathscr T^0_X = \mathcal H om(\Omega_X, \calO_X)$, the tangent sheaf of $X$ \cite[Lemma 3.8]{hacking}. If $p \in X$ and $\pi: Z \to X$ is the canonical covering, then $\mathscr T^i_{QG, X} = (\pi_* \mathscr T^i_Z)^{\mu_N}$. The sheaf $\mathscr T^1_Z$ is supported on the singular locus of $Z$ and $\mathscr T^2_Z$ is supported on the locus where $Z$ is not a local complete intersection. Finally, there is a spectral sequence $$E^{pq}_2 = H^p(\mathscr T^q_{QG,X}) \implies T^{p+q}_{QG, X}$$ given by the spectral sequence for Ext on the canonical covering stack of $X$. 

We now discuss the generalizations of these definitions in the pairs settings (see \cite[Section 3.3]{hacking}).

\begin{definition}\label{def:qgorensteinpair} Let $p \in (X,D)$ be a germ of a stable pair. Let $N$ be the index of $X$ and let $Z \to X$ be the canonical covering. Let $D_Z$ denote the inverse image of $D$. We say a deformation $(\mathcal X, \mathcal D) / (0 \in S)$ of $(X,D)$ is \textbf{$\bQ$-Gorenstein} if there is a $\mu_N$-equivariant deformation $(\mathcal Z, \mathcal D_{\mathcal Z}) / S$ of $(Z,D_Z)$ whose quotient is $(\mathcal X, \mathcal D)/S$. \end{definition}

By Lemma \ref{lemma:gorenstein}, it will follow that our families of pairs  of interest will be $\bQ$-Gorenstein deformations (see \cite[Lemma 3.13]{hacking}). The following theorem will be the main tool we use to show smoothness in the following section.

\begin{theorem}\cite[Theorem 3.12 \& Lemma 3.14]{hacking}\label{thm:hackingunobstruct} Let $(\calX, \calD)/A$ be a $\bQ$-Gorenstein family of stable pairs such that for each closed fiber $(X,D)$, $-K_X$ is ample. Let $A' \to A$ be an infinitesimal extension and $\calX' / A'$ a $\bQ$-Gorenstein deformation of $\calX / A$. Then there exists a $\bQ$-Gorenstein deformation $(\calX', \calD')/A'$ of $(\calX, \calD) / A$ if $H^1(X, \calO_X(D)) = 0$. \end{theorem}

In particular, if one can show that $\bQ$-Gorenstein deformations of $X$ are unobstructed, those of the pair $(X,D)$ are unobstructed as long as $-K_X$ is ample and $H^1(X, \calO_X(D)) = 0$. As a result, the moduli stack parametrizing $\bQ$-Gorenstein deformations of pairs $(X,D)$ will be smooth.

\section{Moduli of del Pezzo surfaces of degree one}\label{sec:dp}

We begin by defining a stable pairs compactification $\mathcal{DP}^1$ of a space of marked degree one del Pezzo surfaces following \cite{hacking}. Later on, we will see that the space $\mathcal{R}(1/12 + \epsilon)$ is a slice inside $\mathcal{DP}^1$, allowing us to apply the methods of \cite{hacking} to $\mathcal{R}(1/12 + \epsilon)$.

\begin{definition}(c.f. \cite[Definition 2.8]{hacking})\label{def:hstable} Let $X$ be a surface and $D$ a $\Q$-Cartier divisor on $X$. Then $(X,D)$ is a \textbf{Hacking stable}, or $H$-stable for short, degree one del Pezzo pair if:
\begin{enumerate}
\item $(X, \left(\frac{1}{12} + \epsilon\right)D)$ is slc and $K_X + (\frac{1}{12} + \epsilon)D$ is ample,
\item the divisor $12K_X + D$ is linearly equivalent to 0, and
\item there is a deformation $(\calX, \calD)/T$ of the pair $(X,D)$ over the germ of a curve such that the general fiber $\calX_t$ is isomorphic to a smooth del Pezzo surface of degree one and the divisors $K_{\calX}$ and $\calD$ are $\Q$-Cartier.
\end{enumerate}
\end{definition}

\begin{remark}A consequence of the above definition is that for an $H$-stable pair, the divisor $-K_X$ is ample (see \cite[Proposition 2.13]{hacking}). In particular, Definition \ref{def:hstable} is an slc generalization of a smooth degree one del Pezzo surface $X$ marked by $D$ the discriminant (weighted with multiplicity) of the anticanonical linear series $|-K_X|$. \end{remark}

We now turn to $\calR(1/12 + \epsilon)$, and show that pairs parametrized by $\calR(1/12 + \epsilon)$ are $H$-stable, so that $\calR(1/12 + \epsilon)$ embeds into $\mathcal{DP}^1$. First, we define a special locus inside $\calR(a)$. 

\begin{definition} Let $\calR^{\circ}(a)$ denote the locus inside $\calR(a)$ parametrizing surfaces without isotrivial $j$-invariant infinity components.  \end{definition}

\subsection{Stable degree one del Pezzo surfaces and $\calR(1/12 + \epsilon)$}\label{sec:a6}

In Section \ref{sec:walls} we computed the type $\mathrm{W}_{\mathrm{III}}$ walls for $\calR(a)$. We see these occur at $1/k$ for $2 \le k \le 9$ (see Theorem \ref{thm:jinfwalls}) and $a_0/k$ for $2 \le k \le 5$ where $a_0$ is a constant appearing in Theorem \ref{thm:transitions} depending on the Kodaira type of the intermediate fiber a Type I pseudoelliptic is attached to (see Theorem \ref{thm:type1walls}). Furthermore, recall the type $\mathrm{W}_{\mathrm{I}}$ walls where fibers become Weierstrass are at $5/6, 3/4, 2/3$, and $1/2$ and type $\mathrm{W}_{\mathrm{II}}$ agree with those of Hassett space (see Remark \ref{rmk:hassettwall2}).

In particular, when $a \leq 1/6$, all sections are contracted so that $\calR(a)$ is a moduli space of pseudoelliptic surfaces. Since the contraction of the section of a rational elliptic surface yields a degree one del Pezzo surface whose pseudofibers are anticanonical curves, we see the following:

\begin{lemma} Let $1/12 < a \le 1/6$. Then $\calR(a)$ is a compactification of a moduli space of degree one del Pezzo surfaces with canonical singularities and marked anticanonical curves. \end{lemma}

We can be more precise about the marking on a del Pezzo surface on the interior of $\calR(a)$. Indeed if $(X,F_a)$ is a normal surface parametrized by a point of $\calR(a)$, then it is the blowdown of the section of a rational Weierstrass fibration. The boundary divisor consists of the singular fibers counted with multiplicity and weighted by $a$. Since each fiber of the Weierstrass fibration becomes an anticanonical curve upon blowing down the section,  we see that 
$$
F_a \sim 12af
$$
where $f \in |-K_X|$ is a pseudofiber class. We may conclude that $F_a \in |-\alpha K_X|$ with $1 < \alpha \le 2$. In particular, the necessarily ample log canonical divisor satisfies 
$$
K_X + F_a \sim_\mathbb{Q} -\delta K_X
$$
for $0 < \delta \le 1$. In particular $(X,F_a)$ can be though of as an anticanonically polarized degree one del Pezzo surface with at worst rational double point singularities.

We now characterize the two types of surfaces parametrized by the boundary of $\calR^0(a)$. 

\begin{theorem}\label{thm:dptypes} The surfaces parametrized by $\calR^{\circ}(a)$ for $1/12 < a \le 1/6$ are either:
\begin{enumerate}
\item normal degree one del Pezzo surfaces with canonical Gorenstein singularities and all singular pseudofibers being Weierstrass of type  $\mathrm{I}_n, \mathrm{II}, \mathrm{III}$ or $\mathrm{IV}$, or
\item the slc union of two degree one del Pezzo surfaces with canonical Gorenstein singularities glued along twisted $\mathrm{I}_0^*$ pseudofibers such that $2K_X$ is Cartier and all other singular pseudofibers as above. 
\end{enumerate}

In both cases, $K_X + D \sim_\bQ -\delta K_X$ for $0 < \delta \le 1$ so that $-K_X$ is ample and $(X,D)$ is an anticanonically polarized. We call case (1) surfaces Type A and case (2) surfaces Type B. 

\end{theorem}

\begin{proof} It follows from Corollary \ref{cor:nocomponents} that the surfaces parametrized by $\calR(1/6 +e)$ have at most two elliptic components. Since $a \le 1/6$ the section of every component contracts. Suppose an elliptic fibration $(X \to C, F_\calA)$ in $\calR^0(1/6 + e)$ has only one elliptic component, possibly with pseudoelliptic components of Type I attached to it. The base rational curve marked by the $(1/6 + \epsilon)$-weighted discriminant $(1/6 + \epsilon)\mathscr{D}$ is an irreducible Hassett stable curve. In particular, the order of vanishing $v_q(\mathscr{D}) \le 5$ for every $q \in C$. Any unstable fiber on the elliptic component is type $\mathrm{II}, \mathrm{III}$ or $\mathrm{IV}$. In particular, any type $I$ pseudoelliptic tree is attached along an intermediate $\mathrm{II}, \mathrm{III}$ or $\mathrm{IV}$ fiber. 

By Theorem \ref{thm:type1walls}, any such pseudoelliptic surface is contracted to a point in the log canonical model for $a \le 1/6$ so every surface in $\calR^0(a)$ arising from a surface in $\calR^0(1/6 + \epsilon)$ with a unique elliptic component is irreducible. Moreover the contraction of the pseudoelliptic components yields singularities of ADE Type By \cite[Page 230]{calculations} since such contractions produce minimal Weierstrass models. The fact that they are del Pezzo surfaces, in the sense that $-K_X$ is ample, follows from calculation preceding the theorem as we saw that $K_X + F_a \sim_{\bQ} -\delta K_X$, which is ample. This gives case (1).  

 Now we discuss Case (2). We saw in Proposition \ref{prop:typeii6} that the only way to obtain multiple elliptic components in $\calR^0(1/6 + \epsilon)$ is if there are two components each with six marked fibers glued along $\mathrm{I_0}^*$ fibers. Again by considering stability of the base Hassett curve, we see that $v_q(\mathscr{D}) \le 5$ so any type $I$ pseudoelliptic trees attached to these components contract to a point by $a = 1/6$. Furthermore the section of each component contracts so we obtain two pseudoelliptic surfaces of Type II glued along twisted $\mathrm{I}_0^*$ fibers but with all other fibers Weierstrass. In particular each component again has only ADE singularities, a single twisted $\mathrm{I}_0^*$ pseudofiber, and else all Weierstrass pseudofibers of types $\mathrm{I}_n, \mathrm{II}, \mathrm{III}$ and $\mathrm{IV}$. Let $(X,F_a,E)$ be such a component with markings $F_a$ and double locus $E$ marked by $1$. Then $F_a$ consists of $6$ (counted with multiplicity) pseudofibers weighted by $a$ and $E$ is a reduced pseudofiber underlying a twisted $\mathrm{I}_0^*$ pseudofiber. Thus as  before we may compute
 $$
 F_a + E \sim_\bQ -(6a + 1/2)K_X < -K_X
 $$
 with $1/12 < a \le 1/6$. Thus $K_X + F_a + E$, the log canonical restricted to each component, satisfies
 $$
 K_X + F_a + E \sim_\bQ -\delta K_X
 $$
 for $\delta > 0$. In particular $-K_X$ is ample and $K_X^2 = 1$ since $-K_X$ is the class of a pseudofiber so each component is a degree one del Pezzo surface. 
 \end{proof}

\begin{lemma}\label{lemma:gorenstein} In the setting above, surfaces of Type A are Gorenstein and surfaces of Type B are $\bQ$-Gorenstein of index 2. \end{lemma}
\begin{proof} The fact that surfaces of Type A are Gorenstein follows from the fact that the singularities are of ADE type (see \cite[Page 230]{calculations}). Surfaces of Type B are Gorenstein away from the double locus as well where the double locus is double normal crossings. Thus we need only check around the points of the double locus where the normalization is singular. There are four such points where each component has an $A_1$ singularity. Locally around each point the surface is a quotient of a nodal (and thus Gorenstein) surface by a $\mathbb{Z}/2\mathbb{Z}$ action since the double locus is a twisted $\mathrm{I}_0^*$. Thus each of these points is $2$-Gorenstein so the whole surface has index $2$. 
 \end{proof}

We have seen already that the surfaces of Type A are anticanonically polarized so it remains to see the same is true for $2$-Gorenstein surfaces of Type B.

\begin{lemma}Surfaces of Type B are anti-canonically polarized. \end{lemma}
\begin{proof} Denote the surface by $X = X_1 \cup X_2$. Let $\nu: \tilde{X} \to X$ denote the normalization, and let $\nu_i$ denote the normalization restricted to the preimage of $X_i$. Then $\nu_i^*(K_X + F_a) = K_{\tilde{X}_i} + \tilde{F_a}|_{X_i} + E$, where $E$ is the preimage of the double locus. We calculate: 
\begin{align*} \nu_i^*(K_X + F_a) &\sim_\bQ -f + 6af + 1/2 f \\ & = -f + 1/2f + \delta f + 1/2f \\ &= \delta f \\ &\sim_\bQ -\delta K_{X_i} \end{align*} 
 for some $\delta >  0$. Here $f$ is a pseudofiber class. On the other hand, $$\nu_i^*(-K_X) = -K_{X_i} - E \sim_\bQ -3/2K_{X_i}.$$ It follows that $K_X + F_a \sim_\bQ -\alpha K_X$ for some $\alpha > 0$, since $\nu^*$ is injective on $\mathrm{Pic } \otimes \bQ$ as the intersection of the irreducible components is a reduced $\bP^1$ (see \cite{mo}) so the pair $(X,F_a)$ is anticanonically polarized. \end{proof}

For the final chamber $a = 1/12 + \epsilon$ such a description actually extends to all of $\calR(1/12 + \epsilon)$: 

\begin{theorem}\label{thm:dpinfinity} The surfaces parametrized by $\calR(1/12 + \epsilon) \setminus \calR^0(1/12 + \epsilon)$ are either the union of 

\begin{enumerate}
\item an isotrivial $j$-invariant infinity surface and a surface of Type A, glued along twisted $\mathrm{N}_1 /\mathrm{I}_n^*$ pseudofibers,
\item or of two isotrivial $j$-invariant infinity surfaces glued along twisted $\mathrm{N_1}$ pseudofibers,
\end{enumerate}

In both cases the surfaces are anticanonically polarized with index $2$. We call the surfaces in (1) Type C and in (2) Type D. 

\end{theorem}

\begin{proof} By examining the twisted stable maps degenerations one sees that the only way to obtain isotrivial components of $j$-invariant $\infty$ is by marked fibers colliding, or a marked fiber colliding with the double locus as in Examples \ref{ex:in} \& \ref{ex:jinfinity} respectively. Any isotrivial components appearing as in Example \ref{ex:in} undergo a pseudoelliptic contraction at $1/k$ for $k = 3, \ldots, 9$ so such components do not appear in the surfaces parametrized by $\calR(1/12 + \epsilon)$. 

Suppose we are in the case of Example \ref{ex:jinfinity}. Then at $a = 1/6 + \epsilon$ as there are only two fibered components $X \cup Y$ along with some trees of type $I$ pseudoelliptics attached. The pseudoelliptics contract at walls $a_0/k$ for $k = 2, \ldots, 5$ and $1/k$ for $k = 3, \ldots, 9$. In particular, all of these components have contracted at $a = 1/12 + \epsilon$. Furthermore, the section of main components $X \cup Y$ contract to type II pseudoelliptics. At least one or both of $X$ and $Y$ are isotrivial $j = \infty$. 

If only one is, suppose $X$, then $X$ has a twisted $\mathrm{N}_1$ fiber attached to a twisted $\mathrm{I}_n^*$ fiber of $Y$ for some $n > 0$. If both are isotrivial $j = \infty$, then they are attached along twisted $\mathrm{N}_1/\mathrm{N}_1$ fibers. Then the corresponding pseudoelliptics are attached along $\mathrm{N}_1/\mathrm{I}_n^*$ respectively $\mathrm{N}_1/\mathrm{N}_1$ pseudofibers. Furthermore, locally around a point of the attaching fiber, by definition of $\mathrm{N}_1/\mathrm{I}_n^*$ fibers, the surface looks like the quotient of a family of nodal curves over a nodal curve modulo a $\mathbb{Z}/2\mathbb{Z}$ action. As a family of nodal curves over a nodal curve is Gorenstein, our surface must be $2$-Gorenstein.  \end{proof}

\begin{theorem}\label{thm:indp1} There is an embedding of $\calR(1/12 + \epsilon)$ into $\mathcal{DP}^1$. Furthermore, the locus $\calR^0(1/12 + \epsilon)$ is a section of its image in the stack of unmarked degree one del Pezzo surfaces under the forgetful morphism $(X,D) \to X$, where $(X,D)$ is an H-stable pair.
\end{theorem}

\begin{proof} Given $(X,F_a) \in \calR(1/12 + \epsilon)$, let $D = 1/aF_a$ where $a = 1/12 + \epsilon$. Then $D$ is a sum of $12$ pseudofibers counted with multiplicty. Let $f$ be a pseudofiber class, then $f \sim_\bQ -K_X$ since $X$ is a pseudoelliptic corresponding to a rational elliptic surface. Thus $12K_X + D \sim_\bQ 0$ verifying condition (2) of Definition \ref{def:hstable}. Condition (1) is true since $(X,F_a)$ is a stable pair and condition (3) follows from the definition of the moduli space $\calR(a)$ as the closure of the component parmetrizing smooth rational (pseudo)elliptic surfaces with only $\mathrm{I}_1$ fibers. 

Finally, over the locus $\calR^0(1/12 + \epsilon)$, the divisor $D$ is the discriminant of the elliptic fibration pushed forward along the pseudoelliptic contraction (excluding the fiber along which two components are glued in the case that $X$ is on the boundary). Thus sending $X$ to the discriminant of its anticanonical pencil gives a section of the projection map with image $\calR^0(1/12 + \epsilon)$ over the locus where $X$ is normal. \end{proof}

\subsection{Smoothness properties of the moduli space $\mathcal{DP}^1$}\label{sec:smooth} Our proof that the moduli space $\mathcal{DP}^1$ is smooth in a neighborhood of $\calR^0(1/12 + \epsilon)$ follows Hacking's thesis \cite{hacking}. In particular, recall that since $-K_X$ is ample, to show that $\calR^0(1/12 + \epsilon)$ is smooth, it suffices to show that the $\bQ$-Gorenstein deformations of the surfaces of Type A and B are unobstructed, and that $H^1(X, \calO_X(D)) = 0$ (see Theorem \ref{thm:hackingunobstruct}).

\begin{prop}(See \cite[Theorem 8.2]{hacking})\label{prop:defa} Let $X$ be a surface of Type A. Then $X$ has unobstructed $\bQ$-Gorenstein deformations. \end{prop}
\begin{proof} Following Hacking, the obstructions are contained in $T^2_{QG, X}$. Since there is a spectral sequence $$E^{pq}_2 = H^p(\mathscr{T}^q_{QG,X}) \implies T^{p+q}_{QG, X},$$ it is sufficient to show that $H^p(\mathscr{T}^q_{QG,X}) = 0$ for $p + q = 2$. 

First note that the sheaf $\mathscr{T}^1_{QG,X}$ is supported on a finite set, since it is supported on the singular locus of $X$, and thus $H^1(\mathscr{T}^1_{QG,X}) = 0$.  The surface $X$ is local complete intersection since the singularities are ADE so $\mathscr T^2_{QG, X} = 0$ and $H^0(\mathscr{T}^2_{QG,X}) = 0$. 

Therefore, it suffices to show that $H^2(\mathscr{T}^0_{QG,X}) = H^2(\mathscr{T}_X) = 0$.  This follows by combining the proof of \cite[Theorem 21]{manetti} and \cite[Lemma 1.11]{steenbrink}. Namely, let $\sigma: S \to X$ be the minimal resolution of $X$. Then since the singularities of $X$ are quotient singularities, $\sigma_* \Omega^1_S = (\Omega_X^1)^{\vee \vee}$ by \cite[Lemma 1.11]{steenbrink}. Therefore, $H^0((\Omega_X^1)^{\vee \vee}) = H^0(\Omega^1_S) = 0$, as $S$ is a rational surface. 

Let $s \neq 0$ be a section of $\calO_X(-K_X)$. Then $s$ yields a dual injective morphism $$s^v: \calO_X(K_X) \to \calO_{X}.$$ Composing with $s^v$ shows that $$\mathrm{Hom}(\mathscr{T}_X, \calO_X(K_X)) = 0$$ and so by Serre Duality ($X$ is Gorenstein!) $H^2(\mathscr{T}_X) = 0$.  \end{proof}

\begin{prop}(see \cite[Theorem 9.1]{hacking})\label{prop:def2} Let $X$ be the a surface of Type B. Then $X$ has unobstructed $\bQ$-Gorenstein deformations. \end{prop}

\begin{proof}
Again, it suffices to show that $H^p(\mathscr{T}^q_{QG,X}) = 0$ for $p + q = 2$. The surface $X$ has local canonical covering by a local complete intersection $\pi : Z \to X$ so that $\mathscr T^2_Z = 0$. If $\mu_n$ is the covering group of $\pi$, we have $\mathscr{T}^2_{QG,Z} = \pi_*(\mathscr T^2_Z)^{\mu_n} = 0$ so $H^0(\mathscr{T}^2_{QG,X}) = 0$. The sheaf $\mathscr{T}^1_{QG,X}$ is supported on the singular locus of $X$ which consists of the pseudofiber along which the surfaces are glued as well isolated ADE singularities. We note that the (induced reduced structure of the) gluing fiber is $\bP^1$, and we let $i: \bP^1 \hookrightarrow X$ denote the inclusion of this fiber in $X$. By \cite[Lemma 3.6]{hasstable}, $\mathscr{T}^1_{QG,X} = i_* \calO_{\bP^1}(1) \oplus \mathcal{Q}$ where $\mathcal{Q}$ is supported at isolated points, and so $H^1(\mathscr{T}^1_{QG,X}) = 0$.

Finally, we must show that $H^2(\mathscr{T}^0_{QG,X}) = H^2(\mathscr{T}_X) = 0$. Let $(X_i, E_i)$ for $i = 1,2$ denote the two components with $E_i = E|_{X_i}$ denoting the restriction of the double locus. Following \cite[Lemma 9.4]{hacking}, to show that $H^2(\mathscr{T}_X) = 0$, it suffices to show that $H^2(\mathscr{T}_{X_i}(-E_i)) = 0$, which is equivalent to showing that $\calO_{X_i}(-K_{X_i} - E_i)$ has a non-zero global section. Note that $-K_{X_i} \sim 2E_i$ since $E_i$ is the support of a multiplicity $2$ nonreduced pseudofiber, and so $-K_{X_i} - E_i \sim E_i$. Thus the reflexive sheaf $\calO_{X_i}(-K_{X_i} - E_i) = \calO(E_i)$ has a section, namely the one cutting out $E_i$. \end{proof}

\begin{lemma}(See \cite[Lemma 3.14]{hacking})\label{lem:h1} Let $(X,D)$ be an $H$-stable pair parametrized by\\ $\cEe \subset \mathcal{DP}^1$. Then $H^1(\calO_X(D)) = 0$. \end{lemma}
\begin{proof} Note that either $X$ or $X^{\nu}$ has canonical singularities (see Theorem \ref{thm:dptypes}). Therefore, it suffices to show that $-(K_X - D)$ is ample, as then the result follows from \cite[Lemma 3.14]{hacking}. Note we have that $12K_X + D \sim_\bQ 0$ since $(X,D)$ is $H$-stable and so $D \sim_\bQ -12K_X$. Thus $-(K_X - D) \sim_\bQ -13K_X$ is ample. \end{proof}

\begin{theorem}\label{thm:12smooth} Let $(X,D) \in \calR^0(1/12 + \epsilon) \subset \mathcal{DP}^1$. Then the stack of $\bQ$-Gorenstein del Pezzo surfaces of degree one is smooth if a neighborhood of $X$ and the projection from $\mathcal{DP}^1$ given by $(X,D) \to X$ is smooth. In particular, $\mathcal{DP}^1$ is smooth in a neighborhood of $\calR^0(1/12 + \epsilon)$, and the locus $\calR^0(1/12 + \epsilon)$ is smooth. \end{theorem} 

\begin{proof} By Propositions \ref{prop:defa} and \ref{prop:def2}, the $\bQ$-Gorenstein deformations of $X$ are unobstructed. By Lemma \ref{lem:h1} and Theorem \ref{thm:hackingunobstruct}, the projection from $\mathcal{DP}^1$ given by $(X,D) \to X$ is a smooth morphism since given a $\bQ$-Gorenstein deformation of $X$, deformations of $D$ are unobstructed. This proves that $\mathcal{DP}^1$ is smooth in a neighborhood of $(X,D) \in \calR^0(1/12 + \epsilon) \subset \mathcal{DP}^1$. Finally, $\calR^0(1/12 + \epsilon)$ is a section of the projection $(X,D) \to X$ over its image by Theorem \ref{thm:indp1} so $\calR^0(1/12 + \epsilon)$ is smooth.  \end{proof} 

\section{Miranda's GIT construction of the moduli space of Weierstrass fibrations}\label{sec:GIT}

\subsection{Overview of Miranda's construction}
In \cite{mir}, Miranda uses GIT to construct a coarse moduli space of \emph{Weierstrass fibrations} (see Definition \ref{def:weierstrassmodel}). Recall these fibrations arise naturally as follows: let $\widetilde{p}: \widetilde{X} \to Y$ be a minimal elliptic surface with section $S$. One obtains a normal surface called a Weierstrass fibration $X \to Y$ by contracting each component of the fibers of $\widetilde{p}$ which do not meet the section $S$. In particular, this fibration has only rational double point singularities, and is uniquely determined by $\widetilde{X}$.

Let $\Gamma_n = \Gamma(\bP^1, \calO_{\bP^1}(n))$. The key point is that $X$ has a \emph{Weierstrass equation}, and as such $X$ can be realized as a divisor in a $\bP^2$-bundle over the base curve. Indeed, for the Weierstrass fibration of a rational elliptic surface, we think of $X$ as being the closed subscheme of $\bP(\calO_{\bP^1}(2) \oplus \calO_{\bP^1}(3) \oplus \calO_{\bP^1})$ defined by the equation $y^2z = x^3 + Axz^2 + Bz^3,$ where $A \in \Gamma_4$, $B \in \Gamma_6$,  and 
\begin{enumerate}
\item $4A(q)^3 + 27B(q)^2 = 0$ precisely at the (finitely many) singular fibers $X_q$, 
\item and for each $q \in \bP^1$ we have $v_q(A) \leq 3$ or $v_q(B) \leq 5$. 
\end{enumerate} 
In this case, constructing a GIT quotient for the moduli space of such surfaces is tantamount to classifying such forms $(A,B) \in \Gamma_4 \oplus \Gamma_6$ subject to these constraints.

Before recalling the geometric characterization of stability that comes from the GIT analysis (see \cite[Theorem 6.2, Section 8, and Section 9]{mir}), we set up some notation. Let $T \subset \Gamma_4 \oplus \Gamma_6$ be the open set of forms satisfying (1) and (2) above. By \cite[Proposition 2.7]{mir}, equivalence classes of Weierstrass fibrations over $\bP^1$ are in 1-1 correspondence with the orbits of $T / (k^* \times \mathrm{SL}(\Gamma_1))$. Let $V \cong T / k^*$. Miranda proves that $V$ is a parameter space for Weierstrass fibrations \cite[Proposition 3.2]{mir}, and so to construct the GIT quotient it suffices to consider $V / \mathrm{SL}(\Gamma_1)$.  We will denote the resulting (coarse) moduli space by $W$. First we recall the stable locus $W^{s}$ of $W$.

\begin{theorem}\cite[Theorem 6.2]{mir}\label{thm:mirstable} Let $r$ be a point of $W$ represented by the pair of forms $(A,B)$ and let $X$ be the rational Weierstrass fibration defined by $(A,B)$. Then $r$ is stable, i.e. $r \in W^{s}$ if and only if $X$ has smooth generic fiber and the associated elliptic surface $\widetilde{X}$ has only reduced fibers. \end{theorem}

Now we discuss the strictly semistable locus $W^{sss} := W \setminus W^s.$

\begin{theorem}\cite[Proposition 8.2 and Theorem 8.3]{mir}\label{thm:mirsss} Let $r$ be a point of $W$ represented by the pair $(A,B)$ and assume that the fibration $X$ defined by $(A,B)$ has smooth generic fiber. Then $r$ is a strictly semistable point, i.e. $r \in W^{sss}$ if and only if the associated elliptic surface $\widetilde{X}$ has a fiber of type $\mathrm{I}_N^*$ for some $N \geq 0$. 

Moreover, two pairs $(A_1, B_1)$ and $(A_2, B_2)$ yielding strictly semistable elliptic surfaces correspond to the same point in $W^{sss}$ if and only if the $j$-invariant of the $\mathrm{I}_N^*$ fibers are the same. \end{theorem}

We remark that this theorem tells us even more (c.f. \cite[Pg. 390]{mir}). Surfaces with a fiber of type $\mathrm{I}_0^*$ are classified in $W^{sss}$ by their $j$-invariant, i.e. there is an $\bA^1 \subset W^{sss}$ parametrizing such surfaces. Surfaces with a fiber of type $\mathrm{I}_N^*$ where $N \geq 1$ are all mapped to a point, as they all have $j$-invariant $\infty$. Finally, if the Weierstrass fibration has no smooth fiber, then it also gets mapped to this $\infty$ point. This allows us to stratify $W$ as follows: $$W = W^{s} \cup \bA^1 \cup \{\infty\}.$$

\subsection{Relation between $W$ and $\cEe$}

We now compare $W^{GIT}$ to $\cEe$, the KSBA compactification of the moduli space of rational elliptic surfaces with twelve $\mathrm{I}_1$ fibers marked with weight $1/12 + \epsilon$. Before proving our main result, we state a crucial lemma from \cite{gg}.

\begin{definition} Let $(A, \mathfrak{m})$ be a DVR with residue field $k$ and fraction field $K$, and let $Y$ be a proper scheme. By the valuative criterion, any map $g: \Spec K \to Y$ extends to a map $\overline{g}: \Spec A \to Y$. We write $\mathrm{lim}(g)$ for the point $\overline{g}(\mathfrak{m}) \in Y$. \end{definition}

\begin{theorem}\label{thm:gg}\cite[Theorem 7.3]{gg} Suppose $X_1$ and $X_2$ are proper schemes over a noetherian scheme $S$ with $X_1$ normal. Let $U \subseteq X_1$ be an open dense set and $f: U \to X_2$ an $S-$morphism. Then $f$ extends to an $S$-morphism $\overline{f}: X_1 \to X_2$ if and only if for any DVR $(A, \mathfrak{m})$ as above and any morphism $g: \Spec K \to U$, the point $\mathrm{lim}(fg)$ of $X_2$ is uniquely determined by the point $\mathrm{lim}(g)$ of $X_1$. \end{theorem}

\begin{theorem}\label{thm:egit} Let $R = R(1/12 + \epsilon)$ be the coarse moduli space of $\calR(1/12 + \epsilon)$ and $\Delta \subset R$ the boundary divisor parametrizing non-normal surfaces with $U = R \setminus \Delta$. There is a morphism $R \to W$ to Miranda's GIT compactification such that the following diagram commutes.
$$
\begin{tikzcd}
\Delta \arrow[r,hook] \arrow{d}[swap]{j} & R \arrow[d] & U \arrow[d,"\cong"] \arrow[l,hook'] \\ 
\mb{P}^1 \arrow[r] & W & W^s \arrow[l,hook']
\end{tikzcd}
$$
Here $\Delta \to \mb{P}^1$ sends the surface $X \cup Y$ to the $j$-invariant of the double locus, $\mb{P}^1 \to W^{sss} \subset W$ maps bijectively onto the strictly semistable locus, and $U \to W^s$ is an isomorphism.
\end{theorem}

\begin{proof}

Let $\calU \subset \cEe$ be the open locus of normal surfaces, i.e. smooth surfaces in $\cEe$ and surfaces of Type A. Consider the $\mathrm{PGL}_2$-torsor $\calF$:
$$ \calF = \{(X, s, t) \mid X \in \calU, (s,t) \in H^0(-K_X) \textrm{ where } s,t \textrm{ span } H^0(-K_X) \}/\sim,$$  where we quotient by scaling. The image of $|-K_X|$ is a $\bP^1$ with coordinates $(s,t)$, and the linear series $|-K_X|$ induces the elliptic fibration: the blowup of its base point gives an elliptic fibration (with section), and thus a Weierstrass equation in coordinates $s$ and $t$. In particular, this Weierstrass coefficients $(A,B)$ are unique up to the scaling of the $\G_m$ action $(A,B) \mapsto (\lambda^4A, \lambda^6B)$.

Furthermore by Theorem \ref{thm:mirstable} and the characterization of surfaces of Type A (Theorem \ref{thm:dptypes}), the forms $(A,B)$ are contained in the stable locus $V^s \subset V$. Therefore, we obtain a $\mathrm{PGL}_2$-equivariant morphism $\calF \to V^s$ which induces a morphism $\phi: \calU \to W$. By comparing the characterization of the type $A$ surfaces parametrized $\calU$ and Miranda's stable surfaces, we see that $\calU \to W$ must be an isomorphism onto the stable locus.

Now suppose that $X$ is a surface parametrized by the boundary $\Delta$ and suppose $\mathscr{X} \to B$ is a 1-parameter family so that $\mathscr{X_b} \in \calU$ for $b \neq 0$ and $\mathscr{X}_0 = X$. Then by Theorem \ref{thm:gg}, to exhibit the existence of a morphism $R \to W$, it suffices to show that $\lim_{b \to 0} (\mathscr{X}_b)$ depends only on $X$ and not on the choice of family. However this follows by Theorem \ref{thm:mirsss}:  the surface $X$ contains a fiber of type $\mathrm{I}_0^*$, namely the gluing fiber, and so the family of Weierstrass data $(A_b,B_b)$ corresponding to $\mathscr{X} \to B$ limits to the unique point $j(I_0^*) \in W^{sss}$ which is well defined since the $j$-invariant of the attaching fiber is the same on each component of $X$. Therefore, the morphism $\phi$ extends to a morphism on all of $R$ and we obtain the desired morphism $R \to W$. Commutativity of the diagram above follows by construction. \end{proof} 
 
 \begin{remark}\label{rem:weierstrass} Given Weierstrass data $(A,B)$, we can consider the discriminant $\mathscr{D} \in \Gamma_{12}$. If $\calD^*$ is the GIT quotient of the space of degree $12$ forms on $\mb{P}^1$ by automorphisms of $\mb{P}^1$, then it is natural to ask if $(A,B) \mapsto \mathscr{D}$ induces a morphism $W \to \calD^*$. There is clearly a rational map $W \dashrightarrow \calD^*$ but this map \emph{cannot} extend. Indeed $W$ parametrizes surfaces with $\mathrm{I}_n$ fibers for $n > 6$ which have discriminant vanishing to order $n > 6$. Such a discriminant is GIT unstable. However, we will see below that the space $R(1/6)$ resolves this rational map. \end{remark} 
 
 \begin{cor}\label{cor:gitbasecurve}  There are morphisms $R(1/6) \to W$ and $R(1/6) \to \calD^*$ where $\calD^*$ is the GIT moduli space for $12$ points in $\mb{P}^1$. Furthermore, the diagram 
 \[
  \begin{tikzcd}
    & R(1/6) \ar{rd} \ar{ld} & \\
    W \ar[rr, dashrightarrow] & & \calD^*
  \end{tikzcd}
  \]
  commutes where $W \dashrightarrow \calD^*$ is the rational map induced by $(A,B) \mapsto \mathscr{D} = 4A^3 + 27B^2$
  \end{cor}
 
 \begin{proof} There is a morphism $R(1/6) \to W$ induced by composing the morphism in Theorem \ref{thm:egit} with the reduction morphism $R(1/6) \to R(1/12 + \epsilon)$ (Theorem \ref{thm:main}). By Theorem \ref{thm:forgetful} there is a morphism $\calR(1/6 + \epsilon) \to \calM_{1/6 + \epsilon}/S_{12}$ which sends a $(1/6 + \epsilon)$-weighted stable rational elliptic surface marked with its singular fibers to the $(1/6 + \epsilon)$-weighted $12$-pointed stable curve marked by the discriminant of the elliptic fibration. By \cite[Section 8]{has}, there is a morphism $\calM_{1/6 + \epsilon}/S_{12} \to \calD^*$ which induces $R(1/6 + \epsilon) \to \calD^*$ by composing and taking coarse moduli space.
 
 To obtain the factorization $R(1/6 + \epsilon) \to R(1/6) \to \calD^*$, it suffices by Theorem \ref{thm:gg} to show that the image of a point under $R(1/6 + \epsilon) \to \calD^*$ depends only on the image of that point in $R(1/6)$. Said another way, we must show that given a $(1/6 + \epsilon)$-weighted broken rational elliptic surface pair, the equivalence class in the GIT moduli space of its discriminant only depends on the $1/6$-weighted stable replacement of the surface.
 
 This is clear on the locus where $R(1/6 + \epsilon) \to R(1/6)$ is an isomorphism. The morphism $R(1/6 + \epsilon) \to R(1/6)$ causes the base curve to contract, and this is an isomorphism on moduli spaces away from the contraction of a trivial $j$-invariant $\infty$ component. By Theorem \ref{thm:jinfwalls}, the only $\wiii$ wall occurring at $a = 1/6$ comes from contracting an isotrivial component glued along an $\mathrm{I}_6$ fiber. In this case the base curve of the corresponding surface parametrized by $R(1/6 + \epsilon)$ had to have two components, each with six marked points. Therefore any surface of this form gets mapped to the unique minimal strictly semistable orbit of $\calD^*$, which arises precisely from two points each of multiplicity six, and so does not depend on the choice of surface. 
 
There also may be the contraction of pseudoelliptic trees of type I to points. However, the discriminant depends only on the main component(s), and not on the pseudoelliptic trees. Indeed since the main components survive under the reduction morphism $R(1/6 + \epsilon) \to R(1/6)$, we see that the stable replacement inside $\calD^*$ only depends on the image of the corresponding point in $R(1/6)$.

Lastly, commutativity is immediate by construction.  \end{proof}

\section{Heckman-Looijenga's compactification} \label{sec:hl}
Recall that to a rational elliptic Weierstrass fibration we can associate its \emph{discriminant divisor} $\mathscr{D}$ which is described in the previous section in terms of Weierstrass equation. Equivalently, for a smooth minimal elliptic surface, $\mathscr{D}$ is given by assigning to any point on the base curve the Euler characteristic of its fiber, and yields an effective divisor of degree 12. When the discriminant is reduced, there are 12 singular fibers of type $\mathrm{I}_1$ -- in this case the projective equivalence class of the discriminant divisor determines the surface up to isomorphism.

As we saw in Section \ref{sec:GIT}, Miranda constructed a compactification  $W$ via geometric invariant theory. Alternatively, we can describe a compactification using the fact that the discriminant can be used to classify generic elliptic surfaces. Using this approach, Heckman and Looijenga showed that the moduli space of rational elliptic surfaces can be interpreted as a locally complex hyperbolic variety,  and studied its Satake-Baily-Borel compactification (see \cite{hl} and \cite[Section 7]{loo}).

Recall that for the GIT compactification $\calD^*$ of the space of 12 points in $\bP^1$ up to automorphism, a collection of points is stable (resp. semistable) if there are no points of multiplicity $\geq 6$ (resp. $\geq 7$). Let $\mathscr{M}$ denote the moduli space of rational elliptic surfaces with \emph{reduced discriminant}, and let $\calD \subset \calD^*$ denote the $\SL_2$ orbit space of 12 element subsets of $\bP^1$. Taking the discriminant of a generic elliptic surface yields a closed embedding $\mathscr{M} \hookrightarrow \calD$ (see \cite[Proposition 2.1]{hl}). While rational elliptic surfaces have 8 dimensional moduli, the dimension of $\calD$ is 9, and so the space of rational elliptic surfaces defines an $\SL_2$-invariant hypersurface. This hypersurface corresponds to the 12 element subsets of $\bP^1$ that admit an equation which is the sum of a cube and a square. 

Heckman and Looijenga obtain a compactification $\mathscr{M}^*$ of the moduli space of rational elliptic surfaces by taking the normalization of the closure of $\mathscr{M}$ inside $\calD^*$. Since they cannot compare $\mathscr{M}^*$ and $W$ directly (see Remark \ref{rmk:hl}), they also define two auxilliary compactifications: $W^*$, which is obtained as the normalization of the closure of the diagonal embedding of $\calM \hookrightarrow W \times \calD^*$, and $\calM^K$ which is a compactification via Kontsevich stable maps . The space $\calM^K$ is essentially the image in the Kontsevich space of maps to $\mb{P}^1$ of the space of twisted stable maps to $\overline{\calM}_{1,1}$ given by composing with the coarse space map $\overline{\calM}_{1,1} \to \mb{P}^1$, see \cite{av, av2,tsm}.

In \cite{hl}, the authors compare the various compactifications and show, using work of Deligne-Mostow \cite{dm}, that $\mathscr{M}^*$ can be interpreted as the Satake-Baily-Borel compactification of a complex hyperball quotient.

\begin{remark}\label{rmk:hl}There are a few points to be about the various compactifications (\cite{hl}). 
\begin{enumerate}
\item The birational map between $W$ and $\mathscr{M}^*$ \emph{does not} extend to a morphism in either direction.
\item Some points of $\mathscr{M}^*$ do not have an interpretation as the isomorphism class of a rational elliptic surface, that is, $\calM^*$ is not the coarse moduli space of some proper Deligne-Mumford stack of elliptic surfaces. 
\item The various compactifications fit together into a diagram as follows: 
 \begin{equation}
  \begin{tikzcd}
    & \calM^K \ar[d] & \\
   & W^* \ar[ld] \ar[rd] & \\
    W \ar[rr, dashrightarrow] & & \mathscr{M}^*
  \end{tikzcd}
  \end{equation}
\end{enumerate}
\end{remark}

The following theorem of \cite{hl} describes the boundary of of $\mathscr{M}$ inside $W, W^*$, and $\mathscr{M}^*$.

\begin{theorem}\cite[Section 3.3]{hl} The boundary of $\mathscr{M}$ inside $W, W^*$, and $\mathscr{M}^*$ is the union of irreducible components denoted by $W(F)$ (resp. $W^*(F)$ and $\mathscr{M}^*(F)$), where $F$ runs over the various Kodaira symbols as in Table \ref{table:sing}. 
\end{theorem}

Here we have $k \leq 5$, and $l,l' \in \{1,2,3,4\}$. The dimension of these components inside each space are in Table \ref{table:sing}. \\

\begin{table}[]
\centering
\caption{Dimension of boundary components}
\label{table:sing}
\begin{tabular}{|l|lll}
\hline
$F$                   & \multicolumn{1}{l|}{$\dim \mathscr{M}^*(F)$} & \multicolumn{1}{l|}{$\dim W(F)$} & \multicolumn{1}{l|}{$\dim W^*(F)$} \\ \hline
$\mathrm{I}_2$        & 7                                            & 7                                & 7                                  \\ \cline{1-1}
$\mathrm{I}_k$        & $9-k$                                        & $9-k$                            & $9-k$                              \\ \cline{1-1}
$\mathrm{I}_6$        & 0                                            & 3                                & 3                                  \\ \cline{1-1}
$\mathrm{I}_7$        & 5                                            & 2                                & 7                                  \\ \cline{1-1}
$\mathrm{I}_8$        & 6                                            & 1                                & 7                                  \\ \cline{1-1}
$\mathrm{I}_9$        & 7                                            & 0                                & 7                                  \\ \cline{1-1}
$\mathrm{II}$         & 7                                            & 7                                & 7                                  \\ \cline{1-1}
$\mathrm{III}$        & 6                                            & 6                                & 6                                  \\ \cline{1-1}
$\mathrm{IV}$         & 5                                            & 5                                & 5                                  \\ \cline{1-1}
$\mathrm{I}^*_0$      & 0                                            & 1                                & 1                                  \\ \cline{1-1}
$\mathrm{I}^*_{l,l'}$ & $l + l' - 1$                                 & 0                                & $l + l' -1$                        \\ \cline{1-1}
\end{tabular}
\end{table}

We now briefly describe the type of surfaces corresponding to the generic point of the boundary loci $W^*(F)$ labeled by Kodaira symbols $F$ in the above theorem (see \cite[Section 3.3]{hl}). \\

\subsubsection{Boundary loci}\label{sec:boundary}
\begin{itemize} \setlength\itemsep{1em}
\item[$\mathrm{I}_{k \geq 2}$] The surface has two components, one isotrivial $j$-invariant $\infty$ component with $k$ marked fibers, glued to a non-isotrivial component along an $\mathrm{I}_k$ fiber. See Example \ref{ex:in}. 
\item[$\mathrm{II}$] The surface has two irreducible components, a $10\mathrm{I}_1$ $\mathrm{II}$ component glued to a $2\mathrm{I}_1$ $\mathrm{II}^*$ component along a $\mathrm{II}/\mathrm{II}^*$ twisted fibers. 
\item [$\mathrm{III}$] Similar to above but with $\mathrm{III}/\mathrm{III^*}$ twisted fibers.
\item [$\mathrm{IV}$] Similar to above but with $\mathrm{IV}/\mathrm{IV^*}$ twisted fibers.
\item[$\mathrm{I}^*_0$] The surface has two irreducible components of type $6\mathrm{I}_1$ $\mathrm{I}_0^*$ and the surfaces are glued along $\mathrm{I}^*_0/\mathrm{I}_0^*$ twisted fibers. Compare with surfaces of Type $A$ in Theorem \ref{thm:dptypes}.
\item[$\mathrm{I}^*_{l,l'}$] The surface has three components $X \cup Y \cup Z$. $Y$ is isotrivial $j$-invariant $\infty$ with $l + l'$ marked nodal fibers as well as two twisted $N_1$ fibers. $X$ has $6 - l$ type $\mathrm{I}_1$ fibers and a twisted $\mathrm{I}_l^*$ glued along one of the $N_1$ fiber and $Z$ is similar with $l'$ instead of $l$. See Example \ref{ex:jinfinity}.
\end{itemize}

\hfill

Roughly speaking, the map $W^* \to W$ takes one of the above surfaces to the equivalence class of semistable orbits in  Miranda's space associated to the Weierstrass equation of the ``main component'' of the surface. Similarly, the map $W^* \to \mathcal{M}^*$ takes such a surface to the GIT semistable replacement of the base curve marked by the discriminant divisor.

\begin{theorem}\label{thm:hlcompare} There is a projective birational morphism $R(1/6) \to W^*$ from the coarse moduli spce of $\calR(1/6)$ which is an isomorphism away from the $W^*(\mathrm{I}_0^*)$ and $W^*(\mathrm{I}^*_{l,l'})$ loci. The universal family of $\calR(1/6)$ over these loci parametrizes surfaces of the type described in Section \ref{sec:boundary}. Furthermore, $\calR(1/6)$ is the minimal space above both $\mathscr{M}^*$ and $W$ extending the universal family on $\mathscr{M}$. \end{theorem}

\begin{proof} First we show that there is a morphism $R(1/6) \to W^*$. $W^*$ is universal for dominant morphisms $X \to W$ and $X \to \mathscr{M}^*$ from a normal variety $X$ that agree over $\mathscr{M}$. By construction $R(1/6)$ is normal and so the existence of $R(1/6) \to W^*$ follows. Let us denote this map by $\varphi$. 

Next one can check by the explicit description of limits in $\calR(1/6)$ given in Sections \ref{sec:walls} and \ref{sec:a6} that $\varphi$ is a bijection over strata $W^*(F)$ for $F \neq \mathrm{I}_0^*, \mathrm{I}_{l,l'}$. Indeed for $F = \mathrm{I}_k$, $2 \le k \le 6$, $\mathrm{II}$, $\mathrm{III}$ and $\mathrm{IV}$ the stratum $W^*(F)$ is the coarse moduli space of Weierstrass surfaces containing an $F$ singular fiber since these correspond to irreducible surfaces whose Weierstrass equation is GIT stable. Thus $\varphi$ is a bijection over these strata as $R(1/6)$ is a coarse moduli space of surfaces parametrized by $\calR(1/6)$. 

The strata $W^*(\mathrm{I}_k)$ for $k = 7,8,9$ parametrize surfaces $X \cup Y$ where $X$ is a trivial surface with $k$ marked fibers and $Y$ is an $\mathrm{I}_k$ Weierstrass surface. The configuration of marked fibers on $X$ is GIT stable in $\calD^*$ and $Y$ is GIT stable in $W$. Therefore $W^*(\mathrm{I}_k)$ is a coarse moduli space for such surfaces. Over this locus $\calR(1/6)$ parametrizes pseudoelliptic models of these same surfaces as in Example \ref{ex:in} and so $\varphi$ is bijective on this locus on the level of coarse moduli spaces. By Zariski's Main Theorem, $\varphi$ is an isomorphism on this locus where it is bijective. 

Over the $W^*(\mathrm{I}_0^*)$, the $\calR(1/6)$ parametrizes pseudoelliptic surfaces $X \cup Y$ glued along a twisted $\mathrm{I}_0^*$ pseudofiber and the map $\varphi$ takes such a surface to the $j$-invariant of the $I_0^*$ fiber. In particular, this stratum in $R(1/6)$ is the $7$-dimensional coarse moduli space for rational elliptic surfaces glued along an $\mathrm{I}_0^*$ fiber while $W^*(\mathrm{I}_0^*)$ is a $1$-dimensional stratum parametrizing only the $j$-invariant. Thus the universal family of $\mathscr{M}$ does not extend over this locus. 

Similarly, over the $W^*(\mathrm{I}_{l,l'}^*)$ locus, $R(1/6)$ is the coarse moduli space for surfaces $X \cup Y \cup Z$ as in Example \ref{ex:jinfinity} and Remark \ref{rmk:jinfinity} where $X$ and $Z$ are $\mathrm{I}_l^*$ and $\mathrm{I}_{l'}^*$ pseudoelliptic surfaces and $Z$ is a chain of isotrivial $j$-infty pseudoelliptic surfaces glued along twisted $\mathrm{N}_1$ fibers. The map $\varphi$ takes such a surface to the GIT semistable replacement of the configuration of marked fibers on the components $Z$. In particular, it forgets the information of $X$ and $Y$ so again the locus $W^*(\mathrm{I}_{l,l'}^*)$ is not a coarse moduli space for the type of surfaces it corresponds to and the universal family over $\mathscr{M}$ does not extend. 

This exhausts the list of strata and shows that $R(1/6) \to W^*$ is an isomorphism away from the locus where $W^*$ is not a coarse moduli space of surfaces. Furthermore, over this locus $R(1/6)$ is a coarse moduli space for precisely the surfaces the strata in $W^*$ correspond to and so $\calR(1/6)$ is the minimal stack over which the universal family of surfaces extends.  \end{proof}

\begin{remark} $R(1/6)$ and $W^*$ are isomorphic along the boundary component corresponding to $\mathrm{I}^*_{4,4}$, but the universal families are different. In the universal family of $\calR(1/6 + \epsilon)$ such surfaces have contracted to a point, but there is a \emph{unique} rational elliptic surface with an $\mathrm{I}^*_4$ fiber (see \cite{persson}). 
\end{remark}

\subsection{Relations to Baily-Borel compactifications}\label{sec:bb}
Using period mappings, Deligne-Mostow show \cite{dm} (see also \cite[Theorem 7.7]{hl}) that there is an isomorphism $\calD^* \cong  (_{\Gamma}\!\backslash\!^{\mathbb{B}})^*$ between the GIT compactification and the Baily-Borel compactification of the ball quotient of a 9-dimensional complex ball.  The rough idea is that given $\mathscr{D}$, one can take the cyclic cover $C \to \bP^1$ of degree 6 totally ramified in $\mathscr{D}$ and study the Jacobian $J(C)$. 

One can ask how the compactification $\mathscr{M}^*$ relates to the quotient $\BB$ introduced above. Indeed, Heckman-Looijenga show that $\mathscr{M}^* \cong\BBO$, a ball quotient obtained by taking the quotient of a (explicit) hyperball $\mathbb{B}_0 \subset \mathbb{B}$. This yields the following isomorphism of arrows \cite[Theorem 9.2]{hl}: 

 \begin{equation}
  \begin{tikzcd}
    \mathscr{M}^* \ar[r, "\cong"] \ar[d] & \BBO \ar[d] \\
    \calD^* \ar[r, "\cong"] & \BB
  \end{tikzcd}
  \end{equation}

Finally, they show (see \cite[Example 10.13]{hl} or \cite[Corollary 7.2]{loo}) that $W^*$ can be realized as the modification (a blowup followed by contraction of a specific locus) of $\BBO$.

\subsection{Relation to Laza-O'Grady}\label{sec:observations}
We urge the reader to consult \cite{ol} and \cite[Section 2]{lo2} for more details. Work of Heckman-Looijenga, and more generally Looijenga studied the problem of understanding the birational map between GIT compactifications ($\calM^{G}$) and Baily-Borel compactifications ($\calM^*$) of large classes of varieties (e.g. rational elliptic surfaces or degree $d$ polarized K3 surfaces with ADE singularities). Baily-Borel showed that $\calM^*$ can be interpreted as $\Proj(R(\calM, \lambda))$, where $\lambda$ is the Hodge bundle. Looijenga then showed that in many cases, the GIT quotient $\calM^{G}$ is realized as $\Proj(R(\calM, \lambda + \Delta))$, where $\Delta$ is some geometrically meaningful divisor. 

Looijenga's insight was that to interpolate between the spaces, one must first $\bQ$-factorialize the divisor $\Delta$, perform a series of specific flips, and then contract the strict transform of $\Delta$.  While this is true for certain cases (e.g. degree 2 K3 surfaces), it is not true for example, in some more involved cases, such as EPW sextics. Instead, Laza-O'Grady (see \cite{ol, lo2, lo3}) initiate the so-called ``Hassett-Keel-Looijenga" program to show that interpolating between $\calM^{G}$ and $\calM^*$ can be seen via varying the coefficient $\beta$ in $\Proj(R(\calM, \lambda + \beta \Delta)$. However, without a modular interpretation, it is hard to analyze the boundary strata and morphisms between these spaces as one varies the coefficient.

In the setting of this paper, we saw that the various KSBA compactifications $\calR(a)$ map to the GIT, Satake-Bailey-Borel, and intermediate compactification $W, \mathscr{M}^*$ and $W^*$ respectively. One may hope that this relation to KSBA is true in other settings studied by Looijenga and Laza-O'Grady. \\

\textbf{Key proposal:} Along similar lines to Laza-O'Grady, who suggest that you can use a $\Proj$ construction to interpolate between GIT and Baily-Borel compactifications, we propose that appropriate KSBA compactifications as one varies the coefficients should interpolate between the GIT and Bailey-Borel compactifications in general. By finding maps from KSBA compactifications to the $\Proj$ constructions appearing above, one can hope to explicitly study and characterize the boundary components of these $\Proj$ in terms of the geometry of stable surface pairs.

\bibliographystyle{alpha}
\bibliography{rational}

\end{document}